\documentclass{amsart}

\usepackage{amsmath,amssymb}
\usepackage{amsthm}
\usepackage{stmaryrd}
\usepackage{enumerate}
\usepackage{url}
\usepackage{wasysym}
\usepackage{color}
\usepackage[normalem]{ulem}

\newtheorem{theorem}{Theorem}[section]
\newtheorem{fact}[theorem]{Fact}
\newtheorem{lemma}[theorem]{Lemma}
\newtheorem{corollary}[theorem]{Corollary}
\newtheorem{proposition}[theorem]{Proposition}
\newtheorem{conjecture}[theorem]{Conjecture}

\theoremstyle{definition}
\newtheorem{definition}[theorem]{Definition}

\newtheorem{remark}[theorem]{Remark}
\newtheorem{question}[theorem]{Question}

\newcommand{\abar}{\bar{a}}
\newcommand{\bbar}{\bar{b}}
\newcommand{\cbar}{\bar{c}}

\newcommand{\xbar}{\bar{x}}
\newcommand{\ybar}{\bar{y}}

\def\seq{\subseteq}
\def\feq{\textit{feq}}
\def\df{\textit{dfs}}
\def\eq{\textit{eq}}
\def\Def{\operatorname{Def}}

\def\tp{\operatorname{tp}}
\def\triv{\operatorname{tr}}
\def\dfs{\operatorname{dfs}}
\def\fa{\operatorname{fam}}
\def\fim{\operatorname{fim}}
\def\Av{\operatorname{Av}}
\def\NIP{\operatorname{NIP}}
\def\NTP{\operatorname{NTP}}
\def\NSOP{\operatorname{NSOP}}
\def\TP{\operatorname{TP}}
\def\N{\mathbb{N}}

\def\R{\mathbb{R}}
\def\fM{\mathfrak{M}}
\def\cL{\mathcal{L}}
\def\cU{\mathcal{U}}

\newcommand{\Fraisse}{Fra\"{i}ss\'{e}}
\newcommand{\claim}{\hfill$\dashv_{\text{\scriptsize{claim}}}$}
\newcommand{\inv}{^{\text{-}1}}

\def\Ind{\setbox0=\hbox{$x$}\kern\wd0\hbox to 0pt{\hss$\mid$\hss}
\lower.9\ht0\hbox to 0pt{\hss$\smile$\hss}\kern\wd0}

\def\Notind{\setbox0=\hbox{$x$}\kern\wd0\hbox to 0pt{\mathchardef
\nn=12854\hss$\nn$\kern1.4\wd0\hss}\hbox to
0pt{\hss$\mid$\hss}\lower.9\ht0 \hbox to 0pt{\hss$\smile$\hss}\kern\wd0}

\def\ind{\mathop{\mathpalette\Ind{}}}

\title{Remarks on generic stability in independent theories}
\author[G. Conant]{Gabriel Conant}

\author[K. Gannon]{Kyle Gannon}

\date{November 25, 2019}

\address{Department of Mathematics\\
University of Notre Dame\\
Notre Dame, IN, 46656, USA}

\email{gconant@nd.edu, kgannon1@nd.edu}

\begin{document}

\begin{abstract}
In NIP theories, generically stable Keisler measures can be characterized in several ways. We analyze these various forms of ``generic stability" in arbitrary theories.  Among other things, we show that the standard definition of generic stability for types coincides with the notion of a frequency interpretation measure. We also give combinatorial examples of types in $\NSOP$ theories that are finitely approximated but not generically stable, as well as $\phi$-types in simple theories that are definable and finitely satisfiable in a small model, but not finitely approximated. Our proofs demonstrate interesting connections to classical results from Ramsey theory for finite graphs and hypergraphs.
\end{abstract}

\maketitle
\vspace{-15pt}

\section{Introduction}

An extremely useful characterization of stability for a complete theory is that any global type is definable and finitely satisfiable in some small model. On the other hand, the class of stable theories is highly restrictive, and a great deal of current research in model theory has focused on finding stable-like phenomena in unstable environments. In NIP theories, although not every type is necessarily definable and finitely satisfiable,  the class of types with these properties is still quite resilient, and such types are now referred to as \emph{generically stable}.

Generically stable types in NIP theories were first identified by Shelah \cite{ShGS}, and then thoroughly studied by Hrushovski and Pillay \cite{HP} and Usvyatsov \cite{UsvyGS}. This investigation was extended to \emph{Keisler measures} in NIP theories  in \cite{HPP} and \cite{HP}, culminating in the work of Hrushovski, Pillay, and Simon \cite{HPS} where \emph{generically stable} Keisler measures were defined. It is shown in \cite{HPS} that a global Keisler measure $\mu$ is definable and finitely satisfiable in a small model if and only if it is \emph{finitely approximated}, i.e., when restricted to  a single formula, $\mu$ is a uniform limit of average frequency measures. Moreover, the ``test points" for these  averages can be chosen from definable sets of almost full measure with respect to products of $\mu$, which leads to the notion of a \emph{frequency interpretation measure} (Definition \ref{def:fim}).

A standard hypothesis in the NIP setting is that definability and finite satisfiability (in a small model) are opposite extremes on the spectrum of invariant types and measures, and so the synthesis of both properties forms a stable refuge in an unstable world. So it is not unreasonable to explore a similar motif beyond NIP theories, and especially in other tame regions like simplicity or $\NTP_2$.

In this paper, we study the above forms of ``generic stability" in the wilderness outside of NIP. Generically stable types in arbitrary theories were defined by Pillay and Tanovi\'{c} in \cite{PiTa} and, in Section \ref{sec:gs}, we reconcile this definition with the setting of measures. Specifically, we show that a global type is generically stable if and only if it is a frequency interpretation measure (Proposition \ref{prop:PTgs}), which establishes a concrete connection between generic stability for measures in NIP theories and for types in arbitrary theories. We also give a characterization of generic stability for types in $\NTP_2$ theories (Theorem \ref{thm:NTP2}), which is formulated purely in terms of forking, and is reminiscent of a similar characterization in the NIP setting. 

In Section \ref{sec:dfs}, we analyze theories in which every \df\ (i.e., definable and finitely satisfiable in a small model) Keisler measure is trivial (we call such theories \emph{dfs-trivial}). We show that  \df-triviality reduces to measures in one variable (Proposition \ref{prop:onevbl}), and that \df-nontriviality is preserved in reducts (Theorem \ref{thm:reduct}). Finally, we give examples of \df-trivial theories, including the theory of the random graph, the theory $T^r_s$ of the generic $K^r_s$-free $r$-hypergraph for $s>r\geq 3$, and the theory $T^*_{\feq}$ of a generic parameterized equivalence relation (Corollary \ref{cor:PIex}).

We then turn to the classes of \df\ measures, \emph{finitely approximated} measures, and \emph{frequency interpretation} measures (listed in increasing strength). As these three classes coincide in NIP theories, we focus on separating them in general theories. For instance, the question of whether finitely approximated measures coincide with frequency interpretation measures in arbitrary theories was asked by Chernikov and Starchenko in \cite[Remark 3.6]{ChStNIP}, and the examples below give a negative answer.

 In Section \ref{sec:ex}, we first recall an example, due to Adler, Casanovas, and Pillay \cite{ACPgs}, of a theory with a generically stable global type $p$ such that $p\otimes p$ is not generically stable. This theory is a variant of $T^*_{\feq}$ in which equivalence classes have size two. We note that this gives a non-simple $\NSOP_1$ theory with a finitely approximated $2$-type that is not generically stable (and thus not frequency interpretable). We then exhibit similar behavior with a $1$-type in the theory  of the generic $K_s$-free graph for $s\geq 3$. Specifically, we consider the global type of a disconnected vertex, which is clearly not generically stable, and use lower bounds on the Ramsey numbers of Erd\H{o}s and Rogers \cite{ErRoR3t} to show this type is finitely approximated. 

At this point, it still remains open whether there is a theory with a definable and finitely satisfiable \emph{global} Keisler measure that is not finitely approximated. However, if we shift our focus to the \emph{local level} of $\phi$-types and $\phi$-measures, then interesting examples emerge. This viewpoint is also motivated by a result of the second author \cite{GannNIP} that if $\phi(x;y)$ is an NIP \emph{formula}, then any definable and finitely satisfiable Keisler measure on $\phi$-definable sets is finitely approximated.  In Section \ref{sec:hyp}, we show that this fails for $\phi$-types in simple theories. In particular, we consider the theory $T^r_s$ for some $s>r\geq 3$, and define  the $\phi$-type $p_R=\{\phi(\xbar;b):b\in\cU\}$ where $\phi(x_1,\ldots,x_{r-1};y)$ is $\neg R(\xbar,y)\wedge\bigwedge_{i\neq j}x_i\neq x_j$. Using the Ramsey property for finite $K^{r-1}_s$-free $(r-1)$-hypergraphs (due to Ne\v{s}et\v{r}il and R\"{o}dl \cite{NesRodHyp}), we show that $p_R$ is finitely satisfiable in any small model. We then show that $p_R$ is not finitely approximated by adapting an averaging argument of Erd\H{o}s and Kleitman \cite{EKcut} on maximal cuts in $(r-1)$-hypergraphs to the setting of weighted hypergraphs. 

A recurring theme in our results is that generic stability in the wild is very uncommon, and more fragile than in NIP theories. Regarding the interaction between \df\ measures and finitely approximated measures, our examples suggest a much weaker connection outside of NIP, at least at the local level. On the other hand, all of our examples of measures that are finitely approximated, but not frequency interpretable, live in theories with $\TP_2$. So perhaps there is  hope for an NIP-like connection for these notions in $\NTP_2$ or simple theories.

\subsection*{Acknowledgements} 
The authors would like to thank Abdul Basit, David Galvin, and Nick Ramsey for helpful discussions. The second author was partially supported by NSF grant DMS-1500671 (Starchenko).

\section{Preliminaries}\label{sec:pre}

Throughout the paper, we let $T$ be a complete $\cL$-theory and $\cU$  a sufficiently saturated monster model of $T$. We write $A\subset \cU$  to mean that $A$ is a subset of $\cU$ and $\cU$ is $|A|^+$-saturated. We write $M\prec\cU$ to mean $M\preceq\cU$ and $M\subset\cU$. Given $A\seq\cU$, we use ``$\cL(A)$-formula" to refer to formulas with parameters from $A$. 

Given a variable sort $x$ and $M\prec\cU$, let $\Def_x(M)$ denote the Boolean algebra of $M$-definable subsets of $\cU^x$, which we also identify with the Boolean algebra of $\cL(M)$-formulas in free variables in $x$. We let $\fM_x(M)$ denote the space of \emph{Keisler measures} (finitely additive probability measures) on $\Def_x(M)$. By viewing types as $\{0,1\}$-valued Keisler measures, we can identify $S_x(M)$ as a closed subset of $\fM_x(M)$.  Also, any Keisler measure in $\fM_x(M)$ can be naturally identified with a regular Borel probability measure on $S_{x}(M)$. See \cite[Section 7.1]{Sibook} for details.

Given a variable sort $x$ and a tuple $\abar=(a_1,\ldots,a_n)\in (\cU^x)^n$, we define the Keisler measure $\Av_{\abar}\in\fM_x(\cU)$ such that, given an $\cL(\cU)$-formula $\phi(x)$,
\[
\Av_{\abar}(\phi(x))=\textstyle\frac{1}{n}|\{i\in[n]:\cU\models\phi(a_i)\}|
\]
(where, by convention, $[n]$ denotes the set $\{1,\ldots,n\}$).

\begin{definition}\label{def:properties}
Fix $\mu\in\fM_x(\cU)$.
\begin{enumerate}
\item $\mu$ is \textbf{invariant} if there is $M\prec\cU$ such that for any $\cL$-formula $\phi(x;y)$ and any $b,b'\in\cU^y$, if $b\equiv_M b'$ then $\mu(\phi(x;b))=\mu(\phi(x;b'))$. In this case, we also say $\mu$ is \textbf{$M$-invariant}.
\item $\mu$ is \textbf{definable} if there is $M\prec\cU$ such that for any $\cL$-formula $\phi(x;y)$ and any closed set $C\seq[0,1]$, the set $\{b\in\cU^{y}:\mu(\phi(x;b))\in C\}$  is type-definable over $M$. In this case, we also say $\mu$ is \textbf{definable over $M$}.
\item $\mu$ is \textbf{finitely satisfiable in $M\prec\cU$}  if for any $\cL(\cU)$-formula $\phi(x)$, if $\mu(\phi(x))>0$ then $\cU\models\phi(a)$ for some $a\in M^x$.  
\item $\mu$ is \textbf{finitely approximated} if there is $M\prec\cU$ such that for any $\cL$-formula $\phi(x;y)$ and any $\epsilon>0$, there is $n\geq 1$ and $\abar\in (M^x)^n$ such that for any $b\in\cU^{y}$, $|\mu(\phi(x;b))-\Av_{\abar}(\phi(x;b))|<\epsilon$. In this case, we call $\abar$ a \textbf{$(\phi,\epsilon)$-approximation for $\mu$}, and we  say $\mu$ is \textbf{finitely approximated in $M$}. 
\end{enumerate}
\end{definition}

The next fact lists some implications between the notions above. These
 are standard exercises (see also \cite[Chapter 7]{Sibook} and \cite[Proposition 4.12]{GannNIP}). 

\begin{fact}\label{fact:props}
Fix $\mu\in\fM_x(\cU)$ and $M\prec\cU$. 
\begin{enumerate}[$(a)$]
\item If $\mu$ is definable over $M$ or finitely satisfiable in $M$, then $\mu$ is $M$-invariant.
\item If $\mu$ is finitely approximated in $M$, then it is definable over $M$ and finitely satisfiable in $M$.
\end{enumerate}
\end{fact}

While Keisler measures were first considered by Keisler in \cite{Keisler}, the notions in Definition \ref{def:properties} were largely developed in later work on NIP theories (see \cite{ChStNIP}, \cite{HPP}, \cite{HP}, \cite{HPS}). In \cite{HPS}, Hrushovski, Pillay, and Simon further introduced  \emph{frequency interpretation measures} in the NIP context. In order to state this definition, we need to recall the notion of the Morley product of two Keisler measures. While in general one can define this product for any pair of measures that are ``measurable with respect to one another'', for simplicity we restrict our focus to definable measures (see Remark \ref{rem:Borel}). 

\begin{definition}
Given $M\prec\cU$, an $\cL(M)$-formula $\phi(x;y)$,  and an $M$-invariant measure $\mu\in \fM_x(\cU)$, define the map $F^\phi_\mu\colon S_y(M)\to [0,1]$ such that $F^\phi_\mu(q)=\mu(\phi(x;b))$ where $b\models q$ (this is well-defined by $M$-invariance of $p$). 
\end{definition}

Definability of measures can be restated as follows.

\begin{fact}
Given $M\prec\cU$, a global $M$-invariant Keisler measure $\mu$ is definable over $M$ if and only if $F^\phi_\mu\colon S_y(M)\to [0,1]$ is continuous for any $\cL(M)$-formula  $\phi(x;y)$.
\end{fact}

We now define the Morley product of  Keisler measures.

\begin{definition}\label{def:prod}
Let $\mu\in \fM_x(\cU)$ and $\nu\in\fM_y(\cU)$ be Keisler measures, and suppose $\mu$ is definable over $M\prec\cU$. We define the product $\mu\otimes\nu$ in $\fM_{xy}(\cU)$ such that, given an $\cL(\cU)$-formula $\phi(x,y)$, 
\[
(\mu\otimes\nu)(\phi(x,y))=\int_{S_y(N)} F^\phi_\mu\, d\nu|_N,
\]
where $N\prec\cU$ contains $M$ and any parameters in $\phi(x,y)$, $F^\phi_\mu$ maps $S_y(N)$ to $[0,1]$, and $\nu|_{N}$ denotes the regular Borel probability measure on $S_y(N)$ associated to the restriction of $\nu$ to $\Def_y(N)$ (we will write $\nu$ instead of $\nu|_N$ when there is no possibility for confusion). 
\end{definition}

One can check that, in the context of the definition, the product $\mu\otimes\nu$ is well-defined and, if $\nu$ is also $M$-invariant, then $\mu\otimes\nu$ is $M$-invariant (see the remarks following \cite[Proposition 7.19]{Sibook}). Also, definability of $\mu$ is more than sufficient to make this definition work (see Remark \ref{rem:Borel}). For NIP theories, it is well known that the Morley product is associative and preserves notions such as definability. While related remarks for arbitrary theories can be found in the folklore, we  take the opportunity to clarify some details regarding our setting of definable measures. 

\begin{proposition}\label{prop:defassoc}$~$
Let $\mu\in\fM_x(\cU)$, $\nu\in\fM_y(\cU)$, and $\lambda\in\fM_z(\cU)$ be Keisler measures, and suppose $\mu$ and $\nu$ are definable over $M\prec\cU$. Then $\mu\otimes\nu$ is definable over $M$, and $\mu\otimes (\nu\otimes\lambda)=(\mu\otimes\nu)\otimes\lambda$.
\end{proposition}

\begin{proof}
We first show $\mu\otimes\nu$ is definable over $M$ (this  is stated without proof in \cite[Lemma 1.6]{HPS}).  Fix an $\cL$-formula $\phi(x,y;z)$. We need to show that the map $F_{\mu \otimes \nu}^{\phi}:S_{z}(M) \to [0,1]$ is continuous. To demonstrate this, we will show that this map is a uniform limit of continuous functions, and hence continuous. 

Fix $\epsilon>0$.  Since $\mu$ is definable,  the map $F_{\mu}^{\phi}:S_{yz}(M) \to [0,1]$ is continuous. Since $S_{yz}(M)$ is a Stone space, there are $\cL(M)$-formulas $\psi_{1}(y,z),\ldots,\psi_{n}(y,z)$, which partition $S_{yz}(M)$,  and real numbers $r_1,\ldots,r_n$ such that  for any $p \in S_{yz}(M)$, $F_{\mu}^{\phi}(p) \approx_{\epsilon} \sum_{i=1}^{n} r_i\chi_{\psi_{i}(y,z)}(p)$ (where $r\approx_\epsilon s$ denotes $|r-s|<\epsilon$, and $\chi_{\psi(y,z)}$ denotes the characteristic function of $\psi(y,z)$ on $S_{yz}(M)$). 
Fix $p\in S_z(M)$, $c\models p|_M$, and $N\prec\cU$ containing $Mc$. Let $\phi^c$ denote $\phi(x,y;c)$ and $\psi^c_i$ denote $\psi_i(y,c)$. Then
\[
F_{\mu \otimes \nu}^{\phi}(p) = \int_{S_{y}(N)} F_{\mu}^{\phi^{c}}\, d\nu \approx_{\epsilon} \int_{S_{y}(N)} \sum_{i =1}^{n} r_i \chi_{\psi^c_{i}(y)}\, d\nu \\
 = \sum_{i=1}^{n} r_i \nu(\psi^c_{i}(y)) = \sum_{i=1}^{n} r_i F_{\nu}^{\psi_{i}}(p).
\]
Since $\nu$ is definable, we have that each $F_{\nu}^{\psi_i}$ is continuous, and so $\sum_{i=1}^{n} r_i F_{\nu}^{\psi_{i}}$ is continuous. Therefore $F_{\mu \otimes \nu}^{\phi}$ is the uniform limit of continuous functions. 

Now, to verify associativity, let $\phi(x,y,z)$ be any $\cL(\cU)$-formula. We define $k_{1} = (\mu \otimes (\nu \otimes \lambda))(\phi(x,y,z))$ and $k_{2} = ((\mu \otimes \nu) \otimes \lambda) (\phi(x,y,z))$, and show $k_1=k_2$.  Let $N\prec\cU$  contain $M$ and any parameters in $\phi(x,y,z)$. Fix $\epsilon>0$, and let $\psi_1(y,z),\ldots,\psi_n(y,z)$ and $r_1,\ldots,r_n$ approximate $F^\phi_\mu\colon S_{yz}(N)\to[0,1]$ as above.  Then
\begin{multline*}
k_1 = \int_{S_{yz}(N)} F_{\mu}^{\phi}\, d(\nu \otimes \lambda) \approx_{\epsilon} \int_{S_{yz}(N)} \sum_{i=1}^n r_i \chi_{\psi_{i}(y,z)}\, d(\nu \otimes \lambda)\\
 = \sum_{i=1}^n r_i (\nu \otimes \lambda)(\psi_i(y,z)) =\int_{S_{z}(N)} \sum_{i=1}^{n} r_i  F_{\nu}^{\psi_i}\,d\lambda.
\end{multline*}
Recall that $k_2 = \int_{S_{z}(N)} F_{\mu \otimes \nu}^{\phi}\, d\lambda$. As above, we have $F_{\mu \otimes \nu}^{\phi}(p) \approx_{\epsilon} \sum_{i=1}^{n} r_i F_{\nu}^{\psi_{i}}(p)$ for any $p \in S_{z}(N)$. Therefore
\[
|k_2 - k_1|  < \int_{S_{z}(N)}\left| F_{\mu \otimes \nu}^{\phi} - {\textstyle \sum_{i=1}^{n} r_i  F_{\nu}^{\psi_i}}\right| d\lambda +\epsilon  <  2\epsilon.\qedhere
\]
\end{proof}

Given $n\geq 1$ and a definable measure $\mu\in\fM_x(\cU)$, let $\mu^{(n)}$ denote the iterated product $\mu\,\otimes\stackrel{n}{\ldots}\otimes\,\mu\in\fM_{\xbar}(\cU)$, where $\xbar=(x_1,\ldots,x_n)$ and $x_i$ is of sort $x$. So $\mu^{(n)}$ is well-defined and definable over $M$ by Proposition \ref{prop:defassoc}. 

\begin{definition}\label{def:fim}
A Keisler measure $\mu\in\fM_x(\cU)$ is a  \textbf{frequency interpretation measure} if for any $\cL$-formula $\phi(x;y)$, there is a sequence $(\theta_n(x_1,\ldots,x_n))_{n=1}^\infty$ of $\cL(\cU)$-formulas satisfying the following properties:
\begin{enumerate}[$(i)$]
\item For any $\epsilon>0$, there is some $n_{\epsilon,\phi}\geq 1$ such that if $n\geq n_{\epsilon,\phi}$, $\abar\models\theta_n(\xbar)$, and $b\in\cU^{y}$, then $|\mu(\phi(x;b))-\Av_{\abar}(\phi(x;b))|<\epsilon$. 
\item $\lim_{n\to\infty}\mu^{(n)}(\theta_n(x_1,\ldots,x_n))=1$.
\end{enumerate}
\end{definition}

In the previous definition, condition $(i)$ implies that $\mu$ is finitely approximated in any $M\prec\cU$  such that the formulas $\theta_n(\xbar)$ in condition $(i)$ are over $M$. In particular, any frequency interpretation measure is definable, and so the iterated product $\mu^{(n)}$ in condition $(ii)$ is well-defined. In NIP theories, the following equivalences hold.

\begin{theorem}[Hrushovski, Pillay, Simon \cite{HPS}]\label{thm:HPS}
Assume $T$ is NIP. Given a Keisler measure $\mu\in\fM_x(\cU)$, the following are equivalent.
\begin{enumerate}[$(i)$]
\item $\mu$ is definable and finitely satisfiable in a small model.
\item $\mu$ is finitely approximated.
\item $\mu$ is a frequency interpretation measure.
\end{enumerate}
\end{theorem}

A related result for types was first proved by Hrushovski and Pillay in \cite{HP} (see Section \ref{sec:gs}). The terminology ``finitely approximated" does not appear in \cite{HPS}, but rather comes from Chernikov and Starchenko \cite{ChStNIP}. 

Finally, we record a few more useful facts about Keisler measures. Given an $\cL$-formula $\phi(x;y)$ we let $\phi^*(y;x)$ be the same formula but with the roles of object and parameters variables exchanged.

\begin{proposition}\label{prop:dfsopp}
Let $\phi(x,y)$ be an $\cL$-formula, and suppose $\mu\in\fM_x(\cU)$ is definable and finitely satisfiable in $M\prec\cU$.
\begin{enumerate}[$(a)$]
\item For any closed set $C\seq [0,1]$, the set $\{b\in\cU^y:\mu(\phi(x;b))\in C\}$ is $\phi^*$-type-definable over $M$.
\item Suppose $b\in\cU^y$ and $\mu(\phi(x;b))>0$. Then there is a $\phi^*$-formula $\psi(y)$, with parameters from $M$, such that $\cU\models\psi(b)$ and $\mu(\phi(x;c))>0$ for any $c\in\psi(\cU)$.
\end{enumerate}
\end{proposition}
\begin{proof}
Let $r_\phi\colon S_y(M)\to S_{\phi^*}(M)$ be the obvious restriction map. Recall that any continuous surjection between compact Hausdorff spaces is a quotient map, and so $r_{\phi}$ is a quotient map. We claim that $F:=F^\phi_{\mu}\circ r\inv_\phi$ is a well-defined function from $S_{\phi^*}(M)$ to $[0,1]$. In other words, we fix $c,c'\in \cU^y$ such that $\tp_{\phi^*}(c/M)=\tp_{\phi^*}(c'/M)$ and show that $\mu(\phi(x;c))=\mu(\phi(x;c'))$. Toward a contradiction, suppose $\mu(\phi(x;c))>\mu(\phi(x;c'))$. Then $\mu(\phi(x;c)\wedge\neg\phi(x;c'))>0$, and thus $\phi(x;c)\wedge\neg\phi(x;c')$ is realized in $M$, which contradicts $\tp_{\phi^*}(c/M)=\tp_{\phi^*}(c'/M)$. 

Since $\mu$ is definable over $M$, we have that $F^\phi_\mu$ is continuous. Now, by the universal property of quotient maps, $F$ is continuous. This immediately implies part $(a)$. For part $(b)$, fix $b\in\cU^y$ such that $\mu(\phi(x,b))>0$. Then $F(\tp_{\phi^*}(b/M))>0$. Fix $0<\delta<F(\tp_{\phi^*}(b/M))$ and consider $U=F\inv((\delta,1])$. Then $U$ is an open set in $S_{\phi^*}(M)$ containing $\tp_{\phi^*}(b/M)$, and so there is a $\phi^*$-formula $\psi(y)$ over $M$ such that $\tp_{\phi^*}(b/M)\in \{p\in S_{\phi^*}(\cU):\psi(y)\in p\}\seq U$. Now $\psi(y)$ is as desired.  
\end{proof}

\begin{proposition}\label{prop:famprod}
Suppose $\mu\in\fM_x(\cU)$, $\nu\in\fM_y(\cU)$, and $M\prec\cU$.
\begin{enumerate}[$(a)$]
\item If $\mu$ and $\nu$ are definable and finitely satisfiable in $M$, then so is $\mu\otimes\nu$. 
\item If $\mu$ and $\nu$ are finitely approximated in $M$, then so is $\mu\otimes\nu$.
\end{enumerate}
\end{proposition}
\begin{proof}
Part $(a)$. By Proposition \ref{prop:defassoc}, $\mu\otimes\nu$ is definable over $M$. We leave finite satisfiability as an exercise. See also \cite[Lemma 2.1]{HPS}, where it is shown that if $\mu$ and $\nu$ are finitely satisfiable in $M$, and $\mu$ is ``Borel definable" (see Remark \ref{rem:Borel}), then $\mu\otimes\nu$ is finitely satisfiable in $M$.

Part $(b)$. Fix $\epsilon>0$, and let $\phi(x,y;z)$ be an $\cL$-formula. Let $\theta_1(x;y,z) = \phi(x,y,z)$ and $\theta_{2}(y;x,z) = \phi(x,y,z)$. Then a straightforward calculation shows that if $\bar{a} \in (M^x)^m$ is a $(\theta_1,\frac{\epsilon}{2})$-approximation for $\mu$ and $\bar{b} \in (M^y)^n$ is a $(\theta_2,\frac{\epsilon}{2})$-approximation for $\nu$, then $((a_i,b_j))_{i\in[m],j\in[n]}\in (M^{xy})^{mn}$ is a $(\phi,\epsilon)$-approximation for $\mu\otimes\nu$.
\end{proof}

The analogue of the previous fact fails for frequency interpretation measures, even in the case of types. This follows from Proposition \ref{prop:PTgs} applied to an
example of Adler, Casanovas, and Pillay \cite{ACPgs} (see Fact \ref{fact:Tfeq2}). 

\begin{remark}\label{rem:Borel}
In Definition \ref{def:prod}, one clearly only needs $F^\phi_\mu$ to be $\nu|_N$-measurable in order for the integral to make sense. If $T$ is NIP, then any $M$-invariant measure is \emph{Borel definable} over $M$, which is to say that for any $\cL$-formula $\phi(x,y)$,  the map $F^\phi_\mu\colon S_y(M)\to [0,1]$ is Borel. So for NIP theories, one only needs invariance of $\mu$ in Definition \ref{def:prod}. In order to give a uniform definition of Morely products in general, one often works with Borel definable measures. So a natural question is whether the analogue of Proposition \ref{prop:defassoc} holds in this general setting (the analogue of part $(a)$ is stated, but not proved, in \cite[Lemma 2.1]{HPS}).

On the other hand, the Morley product of invariant \emph{types} is always well-defined (since every map is measurable with respect to a Dirac measure) and associative. See \cite[Section 2.2]{Sibook} for details.
\end{remark}

\section{Generically stable types}\label{sec:gs}

In NIP theories, a Keisler measure $\mu\in\fM_x(\cU)$ is called \emph{generically stable} if it satisfies the equivalent properties in Theorem \ref{thm:HPS}. Generically stable \emph{types} in NIP theories were initially studied by Shelah \cite{ShGS}, and then in more depth by   Hrushovski and Pillay  \cite{HP} and Usyvatsov \cite{UsvyGS}. In \cite{PiTa}, Pillay and Tanovi\'{c} give a definition of generic stability for types in arbitrary theories (Definition \ref{def:PTgs} below). This notion is further studied by Adler, Casanovas, and Pillay in \cite{ACPgs}. An equivalent formulation of generic stability for types in arbitrary theories is given by Garc\'{i}a, Onshuus, and Usvyatsov in \cite{GOU}.

Given an infinite ordinal $\alpha$ and a sequence $(a_i)_{i<\alpha}$ in $\cU^x$, we let $\Av(a_i)_{i<\alpha}$ denote the \emph{average type} of $(a_i)_{i<\alpha}$ over $\cU$, i.e., the partial type of $\cL(\cU)$-formulas $\phi(x)$ such that $\{i<\alpha:\cU\models\neg\phi(a_i)\}$ is finite.

\begin{definition}[\cite{PiTa}]\label{def:PTgs}
A type $p\in S_x(\cU)$ is \textbf{generically stable} if there is $M\prec\cU$ such that $p$ is $M$-invariant and $\Av(a_i)_{i<\alpha}$ is a complete type for any Morley sequence $(a_i)_{i<\alpha}$ in $p$ over $M$ and any infinite ordinal $\alpha$. In this case, we also say $p$ is \textbf{generically stable over $M$}. 
\end{definition}

We make two remarks. First, the reference to ordinals $\alpha$ other than $\omega$ is necessary in Definition \ref{def:PTgs}. For example,  if $T$ is NIP then any invariant global type satisfies the conclusion of the definition when $\alpha=\omega$, but if $T$ unstable then there is some invariant global type that is not definable (or finitely satisfiable in any small model), and hence not generically stable (see, e.g., \cite[Theorem 2.15]{Pibook}). Second, since Definition \ref{def:PTgs} involves Morley sequences, it does not immediately transfer to measures.\footnote{Randomizations might be a possible future avenue to explore.} The next result clarifies both of these remarks.

\begin{proposition}\label{prop:PTgs}
Given $p\in S_x(\cU)$ and $M\prec\cU$, the following are equivalent.
\begin{enumerate}[$(i)$]
\item $p$ is generically stable over $M$.
\item $p$ is $M$-invariant and $p=\Av(a_i)_{i<\omega}$ for any Morley sequence $(a_i)_{i<\omega}$ in $p$ over $M$.
\item $p$ is a frequency interpretation measure over $M$.
\end{enumerate}
\end{proposition}
\begin{proof}
$(ii)\Rightarrow (i)$. Assume $(ii)$. To show $(i)$, it suffices to consider Morley sequences indexed by $\omega+\omega$ (we leave this as an exercise for the reader). So fix an $\cL(\cU)$-formula $\phi(x)$ and a Morley sequence $(a_i)_{i<\omega+\omega}$ in $p$ over $M$. If $\phi(x)\in p$ then, by $(ii)$, $\{i<\omega:\cU\models\neg\phi(a_i)\}$ and $\{\omega\leq i<\omega+\omega:\cU\models\neg\phi(a_i)\}$ are finite, and so $\{i<\omega+\omega:\cU\models\neg\phi(a_i)\}$ is finite. If $\phi(x)\not\in p$ then $\neg\phi(x)\in p$ and so, by the same reasoning, $\{i<\omega+\omega:\cU\models\phi(a_i)\}$ is finite.

$(i)\Rightarrow (iii)$. Assume $(i)$, and fix an $\cL$-formula $\phi(x;y)$. We construct a sequence $(\theta_n)_{n\geq 1}$ as in Definition \ref{def:fim}. By Definition \ref{def:PTgs} and compactness, there is some $n_\phi$ such that for any Morley sequence $(a_i)_{i<\omega}$ in $p$ over $M$, and any $b\in\cU^{y}$, either $\phi(x;b)\in p$ and $|\{i<\omega:\neg\phi(a_i;b)\}|\leq n_\phi$,  or $\neg\phi(x;b)\in p$ and $|\{i<\omega:\phi(a_i;b)\}|\leq n_\phi$ (see \cite[Proposition 3.2]{HP} or  \cite[Proposition 1]{PiTa} for details). Note that this implies that $p$ is definable over $(a_i)_{i<\omega}$, and thus definable over $M$ by $M$-invariance (see \cite[Lemma 2.18]{Sibook}). So we may choose an $\cL(M)$-formula $\psi(y)$ such that, for all $b\in\cU^{y}$, $\phi(x;b)\in p$ if and only if $\cU\models \psi(b)$.

Given $i\geq 1$, let $n_i=n_\phi i$. We will define a sequence $(\theta_{n_i}(x_1,\ldots,x_{n_i}))_{i=1}^\infty$ of $\cL(M)$-formulas such that, for all $i\geq 1$, $\theta_{n_i}(x_1,\ldots,x_{n_i})\in p^{(n_i)}$ and
\begin{equation*}
\text{if $\abar\models\theta_{n_i}(x_1,\ldots,x_{n_i})$ and $b\in\cU^{y}$ then $|p(\phi(x;b))-\Av_{\abar}(\phi(x;b))|\leq \textstyle\frac{1}{i}$}\tag{$\dagger$}
\end{equation*}
(recall that we view $p$ as a $\{0,1\}$-valued measure in $\fM_x(\cU)$). First, we note that this suffices to prove $(iii)$. Indeed, given $(\theta_{n_i})_{i=1}^\infty$ as above and $n\geq n_\phi$, let $\theta_n(x_1,\ldots,x_n)$ be $\theta_{n_i}(x_1,\ldots,x_{n_i})\wedge\bigwedge_{j\leq n} x_j=x_j$ where $n_i\leq n<n_{i+1}$. Note that $\theta_n(x_1,\ldots,x_n)\in p^{(n)}$ for all $n$. Also, if $n\geq n_\phi$, $\abar\models\theta_n(x_1,\ldots,x_n)$, and $b\in\mathcal{U}$ then, using the triangle inequality and $(\dagger)$, one can show $|p(\phi(x;b))-\Av_{\abar}(\phi(x;b))|< \frac{3}{i}$ where $i$ is such that $n_i\leq n<n_{i+1}$. So it suffices to construct $(\theta_{n_i})_{i=1}^\infty$ as above.

Fix $i\geq 1$, and define the $\cL(M)$-formula
\[
\Phi(x_1,\ldots,x_{n_i};y):=\bigvee_{\substack{I\seq[n_i]\\ |I|>n_\phi}}\left(\bigwedge_{j\in I}\big(\phi(x_j;y)\wedge\neg\psi(y)\big)\vee\bigwedge_{j\in I}\big(\neg\phi(x_j;y)\wedge \psi(y)\big)\right).
\]
Then $p^{(n_i)}(x_1,\ldots,x_{n_i})\wedge\Phi(x_1,\ldots,x_{n_i};y)$ is inconsistent. Setting $\theta_{n_i}(x_1,\ldots,x_{n_i}):=\forall y\neg\Phi(x_1,\ldots,x_{n_i};y)$, we have $\theta_{n_i}(x_1,\ldots,x_{n_i})\in p^{(n_i)}$. It is straightforward to verify that $\theta_{n_i}(x_1,\ldots,x_{n_i})$ satisfies $(\dagger)$.

$(iii)\Rightarrow (ii)$. Assume $(iii)$. As noted after Definition \ref{def:fim}, $p$ is finitely approximated in $M$ and thus $M$-invariant. Fix a Morley sequence $(a_i)_{i<\omega}$ in $p$ over $M$, and some $\phi(x;b)\in p$. Let $I=\{i<\omega:\cU\models\phi(a_i)\}$. By $(iii)$, we may choose $n$ sufficiently large and an $\cL(M)$-formula $\theta(x_1,\ldots,x_n)\in p^{(n)}$ such that, for any $\abar'\models\theta(\xbar)$, $|p(\phi(x))-\Av_{\abar'}(\phi(x;b))|<1$. Note, in particular, that $\theta(a_{i_1},\ldots,a_{i_n})$ holds for any $i_1<\ldots<i_n<\omega$. We now have $|\omega\backslash I|<n$ since, if not, then there are $i_1<\ldots<i_n<\omega$ such that $\neg\phi(a_{i_j};b)$ holds for all $1\leq j\leq n$, and so $p(\phi(x;b))-\Av_{(a_{i_1},\ldots,a_{i_n})}(\phi(x;b))=1$, contradicting the choice of $n$ and $\theta$.
\end{proof}

The previous proposition can be taken as evidence that frequency interpretation measures provide a compatible generalization of the standard notion of generic stability for types to the class of all measures. 

\begin{remark}
Suppose $p\in S_x(\cU)$ is generically stable over $M\prec\cU$, and let $\phi(x;y)$ be an $\cL$-formula. Then we have $\cL(M)$-formulas $(\theta_n)_{n=1}^\infty$ witnessing that $p$ is a frequency interpretation measure over $M$ (as in Definition \ref{def:fim}). By Proposition \ref{prop:dfsopp}, and the proof of Proposition \ref{prop:PTgs}, we see that $\theta_n$ is of the form $\forall y\neg\Phi(x_1,\ldots,x_n;y)$, where $\Phi(x_1,\ldots,x_n;y)$ is a Boolean combination of $\phi(x_i,y)$ and a $\phi^*$-formula $\psi(y)$ over $M$. In particular, if $\cL_0\seq\cL$ contains $\phi(x;y)$, then $p|_{\cL_0}$ is still generically stable over $M$ with respect to $T|_{\cL_0}$. 
\end{remark}

Call a global type $p\in S_x(\cU)$ \emph{stable} over $M\prec\cU$ if $p|_M$ is a \emph{stable type}, i.e., there does not exist a formula $\phi(x;y)$, an $M$-indiscernible sequence $(a_i)_{i<\omega}$ of realizations of $p|_M$, and a sequence $(b_i)_{i<\omega}$ from $\cU^y$ such that $\cU\models\phi(a_i;b_j)$ if and only if $i\leq j$. It is not hard to show that $p\in S_x(\cU)$ is stable over $M\prec\cU$ if and only if $p=\Av(a_i)_{i<\omega}$ for any indiscernible sequence $(a_i)_{i<\omega}$ of realizations of $p|_M$ (see, e.g., \cite{ACPgs}). In particular, if $p\in S_x(\cU)$ is stable over $M\prec\cU$, then it is generically stable over $M$. Using a similar proof (which we leave as an exercise), one obtains an analogous characterization of generic stability in terms of the order property.

\begin{proposition}
Suppose $p\in S_x(\cU)$ is $M$-invariant for some $M\prec\cU$. Then $p$ is generically stable over $M$ if and only if there does not exist a formula $\phi(x;y)$, a Morley sequence $(a_i)_{i<\omega}$ in $p$ over $M$, and a sequence $(b_i)_{i<\omega}$ from $\cU^y$ such that $\cU\models\phi(a_i;b_j)$ if and only if $i\leq j$.
\end{proposition}

For types in arbitrary theories, generic stability is a strengthening of stationarity. In order to make this precise, we recall some definitions.  

\begin{definition}\label{def:stat}
Let $p\in S_x(\cU)$ be a global type and fix $M\prec\cU$.
\begin{enumerate}
\item $p$ is \textbf{stationary over $M$} if, for any $M\seq C\subset\cU$, $p$ is the unique global nonforking extension of $p|_C$. 
\item $p$ is \textbf{weakly stationary over $M$} if $p$ is the unique global nonforking extension of $p|_M$.
\end{enumerate}
\end{definition}

In \cite{HP}, Hrushovski and Pillay give an example of a type $p\in S_x(\cU)$ in an NIP (in fact, C-minimal) theory such that $p$ is weakly stationary over some $M\prec\cU$, but not stationary over $M$. On the other hand, if $T$ is simple then stationarity and weak stationarity are the same by transitivity of nonforking.

\begin{remark}
If $p\in S_x(B)$ does not fork over $C\seq B$ then $p$ has a global extension which does not fork over $C$. In particular, if $p\in S_x(\cU)$ is weakly stationary over $M\prec\cU$ then, for any $B\supseteq M$, $p|_B$ is the unique nonforking extension of $p|_M$ to $B$.
\end{remark}

\begin{fact}\label{fact:gscons}
Suppose $p\in S_x(\cU)$ is generically stable over $M\prec\cU$.
\begin{enumerate}[$(a)$]
\item $p$ is stationary over $M$.
\item If $a\models p|_M$ and $b$ is a tuple from $\cU$ such that $b\ind_M M$, then $a\ind_M b$ if and only if $b\ind_M a$ (where $\ind$ denotes nonforking independence).
\end{enumerate}
\end{fact}
\begin{proof}
Part $(a)$ is in \cite[Proposition 1]{PiTa}. Part $(b)$ is due to Pillay, and the proof is given in \cite{GOU}. See also \cite[Fact 1.9]{ACPgs}.
\end{proof}

\begin{remark}
Given Proposition \ref{prop:PTgs}, a natural question is whether either conclusion of Fact \ref{fact:gscons} holds for finitely approximated types. So we point out this is not the case. In particular, let $T$ be the theory of the generic triangle-free graph, and let $p\in S_1(\cU)$ be the unique  type containing $\neg E(x,b)$ for all $b\in\cU$. By the characterization of forking from \cite{Co13}, $p$ is not weakly stationary over any $M\prec\cU$, and the condition in Fact \ref{fact:gscons}$(b)$ fails for any $M\prec\cU$. However, we will show in Section \ref{sec:Kt} that $p$ is finitely approximated.
\end{remark}

In NIP theories, if $p\in S_x(\cU)$ is stationary over $M\prec\cU$ then it is generically stable over $M$ (see \cite[Proposition 3.2]{HP} or \cite[Theorem 7.6]{UsvyGS}), and so stationarity characterizes generic stability. Our next result uses work of Chernikov and Kaplan from \cite{ChKa} to prove an analogous characterization  for $\NTP_2$ theories. Call a global type $p\in S_x(\cU)$ \textbf{strictly invariant over} $M\prec\cU$ if $p$ is $M$-invariant and $B\ind_M a$ for any $B\supseteq M$ and $a\models p|_B$. Given a model $M\prec\cU$, we say that ``forking equals dividing over $M$" if any \emph{$\cL(\cU)$-formula} that forks over $M$ also divides over $M$.

\begin{lemma}\label{lem:SSI}
Suppose $p\in S_x(\cU)$ is weakly stationary over $M\prec\cU$ and forking equals dividing over $M$. Then $p$ is strictly invariant over $M$. Moreover, if $a\models p|_M$, $b$ is a tuple from $\cU$, and $a\ind_M b$, then $b\ind_M a$. 
\end{lemma}
\begin{proof}
This will be obtained directly from \cite[Proposition 3.7]{ChKa}, and so we will follow the terminology of that paper. In particular, $\ind$ is a standard pre-independence relation satisfying finite character. Moreover, by assumption, $M$ is an extension base for nonforking, and forking implies quasi-dividing over $M$. Altogether, all of the hypotheses of \cite[Proposition 3.7]{ChKa} are satisfied and so we obtain $q\in S_x(\cU)$ such that $q$ extends $p|_M$, $q$ does not fork over $M$, and if $B\supseteq M$ and  $a\models q|_B$ then $B\ind_M a$. By weak stationarity of $p$ over $M$, we have $p=q$. Note also that $p$ is $M$-invariant since it is weakly stationary over $M$. So $p$ is strictly invariant over $M$. Moreover, if $a\models p|_M$, $b$ is a tuple from $\cU$, and $a\ind_M b$, then $a\models p|_{Mb}$ by weak stationarity, and so $b\ind_M a$. 
\end{proof}

\begin{theorem}\label{thm:NTP2}
Assume $T$ is $\NTP_2$. Given $p\in S_x(\cU)$ and $M\prec\cU$, the following are equivalent:
\begin{enumerate}[$(i)$]
\item $p$ is generically stable over $M$;
\item $p$ is weakly stationary over $M$ and, for any $a\models p|_M$ and any tuple $b$ from $\cU$, if $b\ind_M a$ then $a\ind_M b$.
\end{enumerate}
\end{theorem}
\begin{proof}
By \cite[Theorem 1.1]{ChKa}, forking equals dividing over $M$. In particular, $b\ind_M M$ for any tuple $b$, and so we have $(i)\Rightarrow(ii)$ by Fact \ref{fact:gscons}. Now assume $(ii)$. Note that $p$ is $M$-invariant. To show $p$ is generically stable over $M$, it suffices by Proposition \ref{prop:PTgs} to fix a formula $\phi(x;b)\in p$ and a Morley sequence $(a_i)_{i<\omega}$ in $p$ over $M$, and show $\{i<\omega:\cU\models\neg\phi(a_i,b)\}$ is finite. Suppose this fails. After restricting to a subsequence and then replacing it with an $Mb$-indiscernible realization of the EM-type over $Mb$, we have an $Mb$-indiscernible sequence $(a_i)_{i<\omega}$, which is still a Morley sequence in $p$ over $M$ (since $p$ is $M$-invariant), and is  such that  $\cU\models\neg\phi(a_0,b)$.

Let $q(x)=\tp(a_0/Mb)$ and $r(x,y)=\tp(a_0,b/M)$. Then $b$ realizes $\bigcup_{i<\omega}r(a_i,y)$ and $a_i\equiv_M a_0$ for all $i<\omega$. Since $p$ is strictly invariant over $M$ by Lemma \ref{lem:SSI}, and $(a_i)_{i<\omega}$ is a Morley sequence in $p$ over $M$, we may apply   \cite[Lemma 3.14]{ChKa} (which is an analogue of Kim's Lemma for ``strictly invariant" sequences in $\NTP_2$ theories) to conclude $b\ind_M a_0$. So $q$ does not fork over $M$ by $(ii)$.  Now, since $p|_M=q|_M$ and $p$ is weakly stationary over $M$, we must have $q=p|_{Mb}$. But $\neg\phi(x,b)\in q$ and $\phi(x,b)\in p$, which is a contradiction.
\end{proof}

\begin{remark}$~$
\begin{enumerate}[$(1)$]
\item The symmetry assumption in condition $(ii)$ of Theorem \ref{thm:NTP2} is necessary, even for NIP theories, due to the example from \cite{HP} mentioned after Definition \ref{def:stat}.
\item In \cite{SimGSG}, Simon calls a type $p\in S_x(\cU)$ \emph{generically simple over $M\prec\cU$} if $p$ does not fork over $M$ and, for any $a\models p|_M$ and any tuple $b$ from $\cU$, if $b\ind_M a$ then $a\ind_M b$. So  if $T$ is $\NTP_2$, then $p\in S_x(\cU)$ is generically stable over $M\prec\cU$ if and only if it is weakly stationary over $M$ and generically simple over $M$. 
\end{enumerate} 
\end{remark}

\begin{question}
Suppose $T$ is $\NTP_2$ and $p\in S_x(\cU)$ is stationary over $M\prec\cU$. Is $p$ generically simple over $M$?
\end{question}

We also note that if $T$ is simple, then $p\in S_x(\cU)$ is generically stable over $M\prec\cU$ if and only if it is (weakly) stationary over $M$. Moreover, $p\in S_x(\cU)$ is stationary over some $M\prec\cU$ if and only if it has a unique nonforking extension to any larger model. (The right-to-left direction follows from local character and independent amalgamation for forking in simple theories; see, e.g., \cite[Proposition 17.3]{Cabook}.)

It would be interesting to pursue a generalization of Theorem \ref{thm:NTP2} involving frequency interpretation measures in $\NTP_2$ theories. However, this would likely require a better understanding of the general theory of Keisler measures outside of NIP theories, which is still fairly underdeveloped.

\section{dfs-trivial theories}\label{sec:dfs}

We call a global Keisler measure is \textbf{\df} if it is definable and finitely satisfiable in some small model.

\begin{definition}
Fix a variable sort $x$.
\begin{enumerate}
\item Given $a\in\cU^x$, let $\delta_a\in\fM_x(\cU)$ denote the \textbf{Dirac measure on $a$}, i.e., given an $\cL_\cU$-formula $\phi(x)$, $\delta_a(\phi(x))=1$ if and only if $\cU\models\phi(a)$.
\item A measure $\mu\in \fM_x(\cU)$ is \textbf{trivial} if it is in the closure (in the strong topology) of the convex hull of the Dirac measures of points in $\mathcal{U}^{x}$, i.e., there are sequences $(a_n)_{n=0}^\infty$ from $\cU^x$ and $(r_n)_{n=0}^\infty$ from $[0,1]$ such that $\sum_{n=0}^\infty r_n=1$ and $\mu=\sum_{n=0}^\infty r_n\delta_{a_n}$.
\item Let $\fM^{\triv}_x(\cU)$, $\fM^{\dfs}_x(\cU)$, $\fM^{\fa}_x(\cU)$, and $\fM^{\fim}_x(\cU)$ denote the spaces of trivial measures, \df\ measures, finitely approximated measures, and frequency interpretation measures, respectively. 
\item A set $\Omega\seq\fM_x(\cU)$ is \textbf{closed under localization} if, for any $\mu\in\Omega$ and any Borel subset $X\seq S_{x}(\cU)$ with $\mu(X)>0$, $\Omega$ contains the Keisler measure 
\[
\phi(x)\mapsto \mu(\phi(x)\cap X)/\mu(X)
\]
(we call this measure the \textbf{localization of $\mu$ at $X$}).
\end{enumerate}
\end{definition}

Note that, in the last definition above, we have identified $\mu\in\fM_x(\cU)$ with the associated Borel probability measure on $S_x(\cU)$. Note also that a type $p\in S_x(\cU)$ is trivial if and only if it is realized in $\cU$.

\begin{remark}
$\fM^{\triv}_x(\cU)\seq\fM^{\fim}_x(\cU)\seq\fM^{\fa}_x(\cU)\seq\fM^{\dfs}_x(\cU)$,
and each of these sets is closed under localization.
\end{remark}

\begin{proposition}\label{prop:localization}
Suppose $\Omega\seq\fM_x(\cU)$ is closed under localization. Then $\Omega\seq\fM^{\triv}_x(\cU)$ if and only if, for any $\mu\in\Omega$,  there is $b\in\cU^x$ such that $\mu(x=b)>0$. 
\end{proposition}
\begin{proof}
The left-to-right-direction is clear. So assume that for any $\mu\in\Omega$, there is $b\in\cU^x$ such that $\mu(x=b)>0$. Fix $\mu\in\Omega$ and let $S=\{b\in\cU^x: \mu(x=b)>0\}$. 

We first argue that $S$ is countable. Given $b\in X$, let $n(b)\in\N_{\geq 2}$ be such that $\frac{1}{n(b)}<\mu(x=b)\leq\frac{1}{n(b)-1}$. If $S$ is uncountable, then there is some infinite $S_0\seq S$ and $n\geq 2$ such that $n(b)=n$ for all $b\in S_0$. So if $Y\seq S_0$ has size $n$, then $\mu(Y)=\sum_{b\in Y}\mu(x=b)>1$, which is a contradiction. 

Let $\nu=\sum_{b\in S}\mu(x=b)\delta_b$. We will show $\mu=\nu$. First, suppose $X\seq S_x(\cU)$ is Borel and $X\cap S=\emptyset$ (here we identify $\cU$ with the set of realized types in $S_x(\cU)$). Then we claim $\mu(X)=0$. If not, then let $\mu_0\in\fM_x(\cU)$ be the localization of $\mu$ at $X$. Then $\mu_0\in\Omega$, and so there is some $b\in\cU^x$ such that $\mu_0(x=b)>0$, which contradicts $X\cap S=\emptyset$. Now, given a Borel set $X\seq S_x(\cU)$, we have $\mu(X)=\mu(X\backslash S)+\mu(X\cap S)=\mu(X\cap S)=\nu(X)$ as desired. 
\end{proof}

For the rest of this section, we assume $T$ is one-sorted.

\begin{definition}
A complete theory $T$ is \textbf{\df-trivial} if every \df\ Keisler measure is trivial, i.e., $\fM^{\dfs}_n(\cU)=\fM^{\triv}_n(\cU)$ for all $n\geq 1$.
\end{definition}

\begin{proposition}\label{prop:onevbl}
$T$ is \df-trivial if and only if $\fM^{\dfs}_1(\cU)=\fM^{\triv}_1(\cU)$.
\end{proposition}
\begin{proof}
Fix $n\geq 1$ and suppose that every measure in $\fM_n^{\dfs}(\cU)$ is trivial. Suppose $\mu\in\fM^{\dfs}_{n+1}(\cU)$, and let $\mu_0\in \fM_n(\cU)$ be the projection of $\mu$ to the first $n$ variables, i.e., $\mu_0(\phi(x_1,\ldots,x_n))=\mu(\phi(x_1,\ldots,x_n)\wedge x_{n+1}=x_{n+1})$. Note that $\mu_0\in\fM^{\dfs}_n(\cU)$, and thus is trivial by assumption. Fix a countable set $I\subset\cU^n$ and a function $r\colon I\to (0,1]$ such that $\mu_0=\sum_{i\in I}r_i\delta_{i}$. Fix $i\in I$, and let $\nu_i\in\fM_1(\cU)$ be such that $\nu_i(\phi(x))=\frac{1}{r_i}\mu(\phi(x_{n+1})\wedge (x_1,\ldots,x_n)=i)$. Then $\nu_i\in\fM_1^{\dfs}(\cU)$ for all $i\in I$, and so we have $\nu_i=\sum_{j=0}^\infty s^i_j\delta_{a^i_j}$ for some sequences $(a^i_j)_{j=0}^\infty$ from $\cU$ and $(s^i_j)_{j=0}^\infty$ from $[0,1]$. Now we claim that
\[
\mu=\sum_{i\in I}\sum_{j=0}^\infty r_is^i_j\delta_{(i,a^i_j)},
\]
and so $\mu$ is trivial. Let $\xbar=(x_1,\ldots,x_n,x_{n+1})$ and, for $i\in I$, define the formula $\sigma_i(\xbar):=((x_1,\ldots,x_n)=i)\wedge (x_{n+1}=x_{n+1})$. Then $\mu(\sigma_i(\xbar))=r_i$ for any $i\in I$. Since $\sum_{i\in I}r_i=1$, it follows that for any $\cL(\cU)$-formula $\phi(x_1,\ldots,x_n,x_{n+1})$, we have
\begin{multline*}
\mu(\phi(\xbar))=\sum_{i\in I}\mu(\phi(\xbar)\wedge \sigma_i(\xbar))=\sum_{i\in I}\mu(\phi(i,x_{n+1})\wedge \sigma_i(\xbar))= \sum_{i\in I}r_i\nu_i(\phi(i,x))\\
=\sum_{i\in I}\sum_{j=0}^\infty r_is^i_j\delta_{a^i_j}(\phi(i,x))=\sum_{i\in I}\sum_{j=0}^\infty r_is^i_j\delta_{(i,a^i_j)}(\phi(\xbar)).\qedhere
\end{multline*}
\end{proof}

\begin{question}
Does the analogue of Proposition \ref{prop:onevbl} hold for finitely approximable measures or for frequency interpretation measures?
\end{question}

\begin{remark}
If $T$ is NIP then $T$ is \df-nontrivial. This is a standard construction, which we briefly recall (see also, e.g., \cite[Example 7.2]{Sibook}). Assume $T$ is NIP, and let $(a_i)_{i\in [0,1]}$ be a non-constant indiscernible sequence in $\cU$. Define $\mu\in\fM_1(\cU)$ so that $\mu(\phi(x))$ is the Lebesgue measure of $\{i\in [0,1]:\cU\models\phi(a_i)\}$. Since $T$ is NIP, it follows that $\mu$ is a well-defined nontrivial definable Keisler measure, and it is clearly finitely satisfiable in any $M\prec\cU$ containing $(a_i)_{i\in [0,1]}$. 

It should be mentioned that not every NIP theory admits a nontrivial \df\ \emph{type}. For example, in distal theories (which are NIP), any such type must be algebraic (see \cite[Proposition 2.27]{SiD}).  However, if $T$ is stable then any non-algebraic global type is a nontrivial definable and \df\ Keisler measure.
\end{remark}

The next goal is to show that \df-nontriviality is preserved under reducts. First, we make precise our use of the word ``reduct". Let $T_0$ be a complete $\cL_0$-theory in some one-sorted language $\cL_0$ of small cardinality (relative to $\cU$). Without loss of generality, we assume $\cL_0$ is relational. We say $T_0$ is a \textbf{reduct} of $T$ if there is some finite $F\subset\cU$ and, for each $n$-ary relation $R\in \cL_0$,  an $\cL(\cU)$-formula $\theta_R(x_1,\ldots,x_n)$ (with $|x_i|=1$) such that $(\cU\backslash F;(\theta_R)_{R\in \cL_0})\models T_0$. 

\begin{theorem}\label{thm:reduct}
If $T_0$ is a reduct of $T$, and $T_0$ is \df-trivial, then $T$ is \df-trivial.
\end{theorem}
\begin{proof}
Fix a finite set $F\subset\cU$ and $\cL(\cU)$-formulas $(\theta_R)_{R\in\cL_0}$ such that if $\cU_0$ is the $\cL_0$-structure $(\cU\backslash F;(\theta_R)_{R\in\cL_0})$, then $\cU_0\models T_0$. We may add constants to $T$ and assume that each $\theta_R$ is over $\emptyset$, and $F$ is definable by an $\cL$-formula $\chi(x)$ over $\emptyset$.

To show that $T$ is \df-trivial, it suffices by Propositions \ref{prop:localization} and \ref{prop:onevbl} to fix some $\mu\in\fM^{\dfs}_1(\cU)$ and show that $\mu(x=b)>0$ for some $b\in\cU$. Toward a contradiction, suppose $\mu(x=b)=0$ for all $b\in\cU$. Let $M\prec\cU$ be such that $\mu$ is definable and finitely satisfiable over $M$. Then $M_0:=M\backslash F\prec_{\cL_0}\cU_0$, and $\cU_0$ is $|M_0|^+$-saturated as an $\cL_0$-structure. 

Now let $\mu_0$ be the restriction of $\mu$ to $\cL_0$-formulas over $\cU_0$. We show that $\mu_0\in \fM^{\dfs}_1(\cU_0)$, which contradicts the assumption that $T_0$ is \df-trivial. In particular, we show $\mu_0$ is definable (with respect to $\cL_0$) and finitely satisfiable over $M_0$. Fix an $\cL_0$-formula $\phi(x,y)$. Suppose $b\in\cU_0^y$ is such that $\mu_0(\phi(x,b))>0$. Note that $\mu(\chi(x))=0$ by assumption, and so $\mu(\phi(x,b)\wedge\neg\chi(x))>0$. So $\phi(x,b)\wedge\chi(x)$ is realized in $M$, i.e., $\phi(x,b)$ is realized in $M_0$. Now fix a closed set $C\seq[0,1]$ and let $X=\{b\in\cU_0^y:\mu_0(\phi(x,b))\in C\}$. Then $X=\{b\in\cU^y:\mu(\phi(x,b))\in C\}\cap (\cU\backslash F)^y$, and so $X$ is $\cL_0$-type-definable over $M$ by Proposition \ref{prop:dfsopp}$(a)$. Since $F$ is $\emptyset$-definable, $X$ is $\cL_0$-type-definable over $M_0$.
\end{proof}

Recall that the \emph{random graph} is the \Fraisse\ limit of the class of finite graphs, and the \emph{random bipartite graph} is the \Fraisse\ limit of the class of finite bipartite graphs. In order to obtain a \Fraisse\ class in the latter case, we work in the language $\cL=\{E,P,Q\}$ where $E$ is the edge relation and $P,Q$ are predicates for the bipartition. 

\begin{theorem}$~$\label{thm:RGmeas}
\begin{enumerate}[$(a)$]
\item The theory of the random graph is \df-trivial. 
\item The theory of the random bipartite graph is \df-trivial.
\end{enumerate}
\end{theorem}
\begin{proof}
We prove part $(b)$. The argument for part $(a)$ is similar (and easier), so we leave it as an exercise. The case of the random graph is also alluded to by Chernikov and Starchenko in \cite[Example 3.8]{ChStNIP}. 

Let $T$ be the theory of the random bipartite graph. By Propositions \ref{prop:localization} and \ref{prop:onevbl}, it suffices to fix $\mu\in\fM^{\dfs}_1(\cU)$ and show that $\mu(x=b)>0$ for some $b\in\cU$. Toward a contradiction, suppose $\mu(x=b)=0$ for all $b\in\cU$. Fix $M\prec\cU$ such that $\mu$ is definable and finitely satisfiable over $M$. 

Suppose first that there is some $b\in\cU$ such that $\mu(E(x,b))>0$. Without loss of generality, assume $b\in Q(\cU)$. By Proposition \ref{prop:dfsopp}$(b)$, there is a $E^*$-formula $\psi(y)$ over $M$ such that $\cU\models\psi(b)$ and $\mu(E(x,c))>0$ for any $c\in\psi(\cU)$. Without loss of generality, we may assume $\psi(y)$ is of the form
\[
\bigwedge_{m\in A}E(m,y)\wedge\bigwedge_{m\in B}\neg E(m,y)
\]
for some finite disjoint $A,B\seq M$. Note that $A\seq P(M)$. By saturation, there is $c\in \cU$ such that $E(m,c)$ holds for all $m\in A$ and $\neg E(m,c)$ holds for all $m\in M\backslash A$. It follows that $\cU\models\psi(c)$, and so $\mu(E(x,c))>0$. Therefore $\mu(E(x,c)\wedge x\not\in A)>0$. By finite satisfiability, there is some $m\in M\backslash A$ such that $\cU\models E(m,c)$. Then $E(m,c)$ holds and $m\in M\backslash A$, which contradicts the choice of $c$.

Now suppose that $\mu(\neg E(x,b))=1$ for all $b\in\cU$. Note that $\mu(P(x)\vee Q(x))=1$ and so, without loss of generality, we may assume $\mu(P(x))>0$. Let $E_0(x,y)$ be the formula $\neg E(x,y)\wedge ((P(x)\wedge Q(y))\vee (P(y)\wedge Q(x)))$. Then $E_0(x,y)$, $P(x)$, and $Q(x)$ define a random bipartite graph on $\cU$. Moreover, if $b\in Q(\cU)$ then $\mu(E_0(x,b))>0$. So we may apply the argument above to obtain a contradiction.
\end{proof}

We now give several examples of theories which are \df-trivial because they have one of the above theories as a reduct. Recall that, given $r\geq 2$, an \emph{$r$-uniform hypergraph} (or \emph{$r$-graph}) is a set of vertices together with an irreflexive, symmetric $r$-ary relation $R$. For any fixed $r\geq 2$, the class of finite $r$-graphs is a \Fraisse\ class, and we let $T^r$ be the theory of the \Fraisse\ limit. Given $s>r$, let $K^r_s$ be the complete $r$-graph on $s$ vertices. Then, for any fixed $s>r$, the class of finite $K^r_s$-free $r$-graphs is a \Fraisse\ class, and we let $T^r_s$ be the theory of the \Fraisse\ limit.

\begin{corollary}\label{cor:PIex}
The following theories are \df-trivial:
\begin{enumerate}[$(1)$]
\item $T^r$ for any $r\geq 2$,
\item $T^r_s$ for any $s>r\geq 3$, 
\item the theory of any pseudofinite field,
\item the theory of the random tournament,
\item the theory $T^*_{\textit{feq}}$ of a generic parameterized equivalence relation, and
\item any completion of \textnormal{ZF}.
\end{enumerate} 
\end{corollary}
\begin{proof}
$(1)$ We show that $T^2$ is a reduct of $T^r$. Let $F\subset\cU$ be a set of size $r-2$. Let $E(x,y)$ be $R(x,y,\cbar)$ where $\cbar$ enumerates $F$. Suppose $A,B\subset\cU\backslash F$ are finite and disjoint. Define a one-point extension $H=ABF\cup \{e\}$ of the induced hypergraph on $ABF$ by adding only the edges $R(a,e,\cbar)$ for all $a\in A$. Then $H$ is an $r$-graph, and so we may assume $H$ embeds in $\cU$ over $ABF$. Now $E(a,e)$ holds for all $a\in A$ and $\neg E(b,e)$ holds for all $b\in B$.

$(2)$ Note that if $r\geq 3$ and $\cU\models T^r_s$ then the graph $H$ constructed in $(1)$ is $K^r_s$-free. So the same argument works to show that $T^2$ is a reduct of $T^r_s$.

$(3)$ Let $T$ be the theory of a pseudofinite field $K$, and let $p$ be a prime different from the characteristic of $K$. By a result of Duret \cite{Duret}, $T^2$ is a reduct of $T$ via the formula $E(x,y):= \exists z(x+y=z^p)\wedge x\neq y$. 

$(4)$ Recall that a tournament is a directed graph in which every pair of vertices is joined by exactly one directed edge. Let $T$ be the theory of the random tournament (i.e., the \Fraisse\ limit of finite tournaments), and let $\cU\models T$. We show that the theory of the random bipartite graph is a reduct of $T$.  Let $R$ be the directed edge relation, and fix some $a\in\cU$. Let $P=\{b\in\cU:R(a,b)\}$ and let $Q=\{b\in\cU:R(b,a)\}$. Note that $P$ and $Q$ partition $\cU\backslash\{a\}$. Define a bipartite graph relation $E\seq P\times Q$ where $E(b,c)$ holds if and only if $R(b,c)$. Then $(P,Q;E)$ satisfies the axioms of the random bipartite graph. 

$(5)$ We show that the theory of the random bipartite graph is a reduct of $T^*_{\textit{feq}}$ (see Section \ref{sec:Tfeq2} for the definition of this theory). Let $E_z(x,y)$ be the parameterized equivalence relation, where $x,y$ are in the object sort $O$ and $z$ is in the parameter sort $P$. Fix some $a\in O(\cU)$, and let $P_0=O(\cU)\backslash\{a\}$ and $Q_0=P(\cU)$. Define a bipartite graph relation $E_0\seq P_0\times Q_0$ where $E_0(b,c)$ holds if and only if $E_c(a,b)$. Then $(P_0,Q_0;E_0)$ satisfies the axioms of the random bipartite graph. 

$(6)$ Let $T$ be a completion of ZF, and let $\cU\models T$. We show that $T^2$ is a reduct of $T$ via the formula $E(x,y):=(x\in y\vee y\in x)$.\footnote{This was observed by James Hanson.} Fix finite disjoint $A,B\subset\cU$. Define $c=A\cup\{B\}$, which is an element of $\cU$. Then $E(a,c)$ holds for all $a\in A$ and, by the axiom of foundation, we have $\neg E(b,c)$ for all $b\in B$. 
\end{proof}

Noticeably absent from the previous corollary is $T^2_s$ for $s\geq 3$. We will see in Section \ref{sec:Kt} that these theories are \emph{not} \df-trivial.

\section{Examples}\label{sec:ex}

\subsection{Parameterized equivalence relations}\label{sec:Tfeq2} The purpose of this section is to develop an example of Adler, Casanovas, and Pillay \cite{ACPgs}. Let $\cL$ be a language with two sorts $O$ and $P$ (for ``objects" and ``parameters") and a ternary relation $E_z(x,y)$ on $O\times O\times P$ (with $x,y$ of sort $O$ and $z$ of sort $P$). Let $T_{\feq2}$ be the incomplete theory asserting that for any $z$ in $P$, $E_z(x,y)$ is an equivalence relation on $O$ in which each class has size $2$. Then $T_{\feq2}$ has a model completion, which we denote $T^*_{\textit{feq}2}$. This theory was defined in  \cite[Example 1.7]{ACPgs}, and can also be constructed as the \emph{generic variation} of the theory $T^*_ {\textit{eq}2}$ of an equivalence relation with infinitely many classes of size $2$. Generic variations were defined by Baudisch in \cite{BauGV}, although we have used an equivalent two-sorted version as in \cite[Section 6.2]{ChRa}. 

Note that $T^*_{\textit{eq}2}$ has quantifier elimination and eliminates $\exists^\infty$ (in fact, this theory is complete and strongly minimal). It follows that $T^*_{\feq2}$ is complete, model complete, and eliminates $\exists^\infty$ (see \cite[Corollary 2.10, Theorem 3.1]{BauGV}).  However, $T^*_{\feq2}$ does not eliminate quantifiers unless one adds a binary function $f\colon P\times O\to O$ such that, for any $z\in P$, $f_z(-)\colon O\to O$ swaps the the two elements in each $E_z$-class (more precisely, $f_z(x)=y$ if and only if $E_z(x,y)\wedge x\neq y$). 

The classification of $T^*_{\feq2}$ using dividing lines is the same as its counterpart $T^*_{\feq}$. Recall that $T^*_{\feq}$ is the generic variation of the theory $T^*_{\textit{eq}}$ of an equivalence relation with infinitely many infinite classes (or, alternatively, the model companion of the theory $T_{\feq}$ of a parameterized equivalence relation). 

\begin{theorem}
$T^*_{\textit{feq}2}$ is $\NSOP_1$ but not simple.
\end{theorem}
\begin{proof}
To show that $T^*_{\textit{feq}2}$ is not simple we will in fact witness $\TP_2$ (recall the result of Shelah that any $\NSOP_1$ non-simple theory must have $\TP_2$ and, as is often the case in such situations, a direct demonstration of $\TP_2$ is cleaner).

Let $\{b_{i,j},c_i:i,j<\omega\}\seq O(\cU)$ be a collection of pairwise distinct objects. Given $i<\omega$ and $j<k<\omega$, the formula $E_x(b_{i,j},c_i)\wedge E_x(b_{i,k},c_i)$ is inconsistent since all $E_x$-classes have size $2$. On the other hand, for any function $\sigma\colon\omega\to\omega$, the type $\{E_x(b_{i,\sigma(i)},c_i):i<\omega\}$ is consistent since we can find a parameter $a\in P(\cU)$ such that $\{b_{i,\sigma(i)},c_i\}$ is an $E_a$-class for all $i<\omega$. Altogether, we have shown that the formula $\phi(x;y_1,y_2):=E_x(y_1,y_2)$ has $\TP_2$.

Finally, the fact that $T^*_{\feq2}$ is $\NSOP_1$ follows from an unpublished result of Ramsey that $\NSOP_1$ is preserved by the generic variation construction from \cite{BauGV}  (the analogous statement for $\NTP_1$ was shown by Dobrowolski in \cite{DobGV}). 
\end{proof}

\begin{remark}
In \cite[Section 6.2]{ChRa}, Chernikov and Ramsey consider the theories of \Fraisse\ limits in finite relational languages with strong amalgamation, and they show that generic variation preserves $\NSOP_1$ in such cases. Note however that, while $T^*_{\eq2}$ is the theory of a \Fraisse\ limit in a finite relational language, the associated \Fraisse\ class does not have strong amalgamation. 
\end{remark} 

Despite the similarities between the definitions of $T^*_{\feq 2}$ and $T^*_{\feq}$, and the dividing lines they satisfy, the behavior of generically stable types is different (recall that $T^*_{\feq}$ is \df-trivial by Corollary \ref{cor:PIex}). 

\begin{remark}\label{rem:fincof}
If $\cU\models T^*_{\feq2}$ then any definable subset of $O(\cU)$ is finite or cofinite (this is easily checked using quantifier elimination in the language with $f$ named).
\end{remark}

\begin{fact}[Adler, Casanovas, \& Pillay \cite{ACPgs}]\label{fact:Tfeq2}
If $\cU\models T^*_{\feq2}$ then there is a generically stable type $p\in S_1(\cU)$ such that $p\otimes p$ is not generically stable.
\end{fact}
\begin{proof}
See \cite[Example 1.7]{ACPgs}. The type $p$ is the unique type in $S_1(\cU)$ that contains $O(x)$ and $\neg E_c(x,b)$ for all $b\in O(\cU)$ and $c\in P(\cU)$. By Remark \ref{rem:fincof}, it is clear that $p$ is generically stable, and it is shown in \cite{ACPgs} that $p\otimes p$ is not generically stable.
\end{proof}

Since the type $p$ in the previous fact is not realized in $\cU$, we conclude that $T^*_{\feq2}$ is not \df-trivial. We also obtain a separation between generic stability and finite approximation for types (recall that these notions are equivalent in $\NIP$ theories).

 \begin{corollary}
 If $\cU\models T^*_{\feq2}$ then there is a type $q\in S_2(\cU)$ that is finitely approximated but not generically stable. 
 \end{corollary}
 \begin{proof}
Let $p\in S_1(\cU)$ be the type from Fact \ref{fact:Tfeq2}. Then $p\in \fM_1^{\fim}(\cU)$ by Proposition \ref{prop:PTgs}. So $q:=p\otimes p\in \fM_2^{\fa}(\cU)$, since $\fM_1^{\fim}(\cU)\seq\fM_1^{\fa}(\cU)$ and finitely approximated measures are closed under Morley products (see Proposition \ref{prop:famprod}$(b)$).
 \end{proof}

\subsection{$K_s$-free graphs}\label{sec:Kt}

Fix $s\geq 3$ and let $K_s$ be the complete graph on $s$ vertices. 
Given a finite graph $G$, let $\alpha_s(G)$ denote the size of the largest subset of $G$ which induces a $K_{s-1}$-free subgraph. Let $R_s(n)$ be the smallest integer $N$ such that any graph $G$ of size $N$ either contains $K_s$ or satisfies $\alpha_s(G)\geq n$.

\begin{theorem}[Erd\H{o}s \& Rogers 1962 \cite{ErRoR3t}]\label{thm:ER3t}
$R_s(n)\geq \Omega(n^{1+c_s})$ for some $c_s>0$. Thus there are  $K_s$-free graphs $(G_i)_{i=0}^\infty$ such that $|G_i|\to\infty$ and $\alpha_s(G_i)= o(|G_i|)$.
\end{theorem}

 \begin{remark}
For $s=3$, Theorem \ref{thm:ER3t} was first proved  by Erd\H{o}s \cite{ErdR3t} in 1957, and it was eventually shown that $R_3(n)=\Theta(\frac{n^2}{\log n})$ (see \cite{AKSR3t} and \cite{KimR3t}).  
 \end{remark}
 
 We use $\cL=\{E\}$ for the language of graphs, and let $\cU\models T^2_s$ be sufficiently saturated. By quantifier elimination for $T^2_s$, there is a unique type in $S_1(\cU)$ containing $\neg E(x,b)$ for all $b\in\cU$. We let $p_E$ denote this type. Note that $p_E$ is definable over $\emptyset$ and not realized in $\cU$.

\begin{theorem}\label{thm:famnotfim}
The type $p_E$ is finitely approximated, but is not generically stable.
\end{theorem}
\begin{proof}
Let $(a_i)_{i<\omega}$ be a Morley sequence in $p_E$ over some small model. Then $a_i\neq a_j\wedge\neg E(a_i,a_j)$ holds for all $i \neq j$, and so there is $b\in \cU$ such that $E(a_i,b)$ holds if and only if $i$ is even. So $p_E$ is not generically stable.

Now we show that $p_E$ is finitely approximated in the unique countable model $M$ of $T^2_s$.
Let $\phi(x;\ybar)$ be a formula in the language of graphs, with $\ybar=(y_1,\ldots,y_m)$. Without loss of generality, we may assume that some instance of $\phi(x;\ybar$) is in $p_E$ (otherwise every instance of $\neg\phi(x;\ybar)$ is in $p_E$, and so we may apply the argument below to $\neg\phi(x;\ybar)$). 

By quantifier elimination, we may fix a quantifier-free formula $\psi(\ybar)$, an integer $N\geq 1$, and $A_t,B_t,C_t,D_t\seq [m]$, for $t\in[N]$, such that
\[
\phi(x;\ybar)\equiv\bigvee_{t=1}^N\left(\bigwedge_{i\in A_t}\neg E(x,y_i)\wedge\bigwedge_{i\in B_t}x\neq y_i\wedge \bigwedge_{i\in C_t}E(x,y_i)\wedge \bigwedge_{i\in D_t}x=y_i\right)\wedge\psi(\ybar).
\]
Since $p_E$ contains an instance of $\phi(x;\ybar)$, it follows that there is some $t_*\in[N]$ such that $C_{t_*}=\emptyset=D_{t_*}$, and so $\phi(x;\bbar)\in p_E$ for any $\bbar\in \cU^m$ such that $\cU\models\psi(\bbar)$.  Fix $\epsilon>0$. We want to find $n\geq 1$ and $\abar=(a_1,\ldots,a_n)\in M^n$ such that, for any $\bbar\in \cU^m$, 
\begin{equation*}
\left|p(\phi(x;\bbar))-\Av_{\abar}(\phi(x;\bbar))\right|<\epsilon.\tag{$\dagger$}
\end{equation*}

Let $|A_{t_*}|=k$ and $|B_{t_*}|=\ell$. By Theorem \ref{thm:ER3t}, we may choose $n>\frac{2\ell}{\epsilon}$ and $G=\{a_1,\ldots,a_n\}\subset M$ such that $\alpha_s(G)<\frac{\epsilon}{2k}n$. Fix $\bbar\in \cU^m$. If $\cU\models\neg\psi(\bbar)$ then $\phi(x;\bbar)\not\in p$ and $\cU\models\neg\phi(a_i;\bbar)$ for all $i\in[n]$, so $(\dagger)$ holds trivially. So we can assume $\cU\models\psi(\bbar)$, which implies that $\phi(x;\bbar)\in p$. 

For $j\in A_{t_*}$, set $X_j=\{i\in[n]:E(a_i,b_j)\}$, and note that $\{a_i:i\in X_j\}$ induces a $K_{s-1}$-free subgraph of $G$ (since $\cU$ is $K_s$-free). In particular $|X_j|<\frac{\epsilon}{2k}n$ for all $j\in A_{t_*}$. Define the sets $Y=\{i\in[n]:a_i=b_j\text{ for some }j\in B_{t_*}\}$ and $Z=\{i\in[n]:\neg\phi(a_i;\bbar)\}$. Then we have $Z\seq Y\cup \bigcup_{j\in A_{t_*}}X_j$, which implies
\[
|Z|\leq |Y|+\sum_{j\in A_{t_*}}|X_j|< \ell+\textstyle\frac{\epsilon}{2}n <\epsilon n.
\]
So $(\dagger)$ holds, as desired.
\end{proof}

\begin{remark}
From the proof of Theorem \ref{thm:famnotfim} we see that, given an $\cL$-formula $\phi(x;\ybar)$, there is a sequence $(\theta_n(x_1,\ldots,x_n))_{n=1}^\infty$ of $\cL$-formulas (over $\emptyset$) such that, for any $\epsilon>0$, if $n\geq n_{\epsilon,\phi}$ then $|p_E(\phi(x;\bbar))-\Av_{\abar}(\phi(x;\bbar))|<\epsilon$ for any $\abar\models\theta_n(\xbar)$ and $\bbar\in\cU^{\ybar}$. In particular, let $\theta_n(x_1,\ldots,x_n)$ describe the isomorphism type of the graph $G$ chosen with $\alpha_s(G)$ sufficiently small depending on $\epsilon$ and $\phi$. Of course, since $\alpha_s(G)$ is small, $G$ must contain (many) edges, and so $\theta_n(\xbar)\not\in p^{(n)}_E$.
\end{remark}

Next, we show that if $\cU\models T^2_s$ then $\fM_1^{\dfs}(\cU)$ coincides with $\fM_1^{\fa}(\cU)$, and is the convex hull of $p_E$ and $\fM_1^{\triv}(\cU)$. So $p_E$ is essentially the only non-trivial \df\ measure in one variable. We also observe that every frequency interpretation measure in one variable is trivial.

\begin{theorem}\label{thm:Ktmeas} Let $\cU\models T^2_s$.
\begin{enumerate}[$(a)$]
\item $\fM_1^{\fa}(\cU)=\fM_1^{\dfs}(\cU)=\{rp_E+(1-r)\mu:\mu\in\fM_1^{\triv}(\cU),~r\in[0,1]\}$.
\item $\fM_1^{\fim}(\cU)=\fM^{\triv}_1(\cU)$.
\end{enumerate}
\end{theorem}
\begin{proof}
Part $(a)$. Let $\mu\in\fM_1(\cU)$ be definable and finitely satisfiable over $M\prec\cU$. 
\vspace{5pt}

\noindent\emph{Claim:} If $\mu(x=b)=0$ for all $b\in\cU$, then $\mu=p_E$.

\noindent\emph{Proof:}
Assume $\mu(x=b)=0$ for all $b\in\cU$ and, toward a contradiction, suppose $\mu(E(x,b))>0$ for some $b\in\cU$. There are two cases. 

Suppose first that $b\not\in M$. Let $\psi(y)$ be an $\cL(M)$-formula such that $\psi(b)$ holds and, for any $c\in\cU$, if $\psi(c)$ holds then $\mu(E(x,c))>0$. Let $A\subset M$ be the finite set of parameters in $\psi(y)$. Since $b\not\in M$, we may find $c\in\cU$ such that $c\equiv_A b$ and $\neg E(m,c)$ for all $m\in M\backslash A$. Then $\psi(c)$ holds and so $\mu(E(x,c)\wedge x\not\in A)>0$. But $E(x,c)\wedge x\not\in A$ is not realized in $M$.

Now suppose $b\in M$. Let $X=\{m\in M:E(m,b)\}$. Then $X$ is $K_{s-1}$-free, so there is $c\in\cU\backslash M$ such that $E(m,c)$ for all $m\in X$. By the above, $\mu(\neg E(x,c))=1$, and so $\mu(\neg E(x,c)\wedge E(x,b))>0$. But $\neg E(x,c)\wedge E(x,b)$ is not realized in $M$.\claim
\vspace{5pt}

Now, let $S=\{b\in\cU:\mu(x=b)>0\}$. As in the proof of Proposition \ref{prop:localization}, $S$ is countable. By the claim, we may assume $S\neq\emptyset$, and so $\mu(S)>0$. Let $\nu=\frac{1}{\mu(S)}\sum_{b\in S}\mu(x=b)\delta_b$, and note that $\nu\in\fM_1^{\triv}(\cU)$. 

        Let $X=\cU\backslash S$. If $\mu(X)=0$ then $\mu=\nu$, and we are finished. So assume $\mu(X)>0$. Let $\mu_0$ be the localization of $\mu$ at $X$. Then $\mu_0$ is \df\ and $\mu_0(x=b)=0$ for all $b\in\cU$. By the claim, $\mu_0=p_E$. Note that $p_E(S)=0$ since $S$ is countable. Altogether, given $A\in\Def_1(\cU)$, we have
\begin{multline*}
\mu(A)=\mu(A\backslash S)+\mu(A\cap S)\\
=\mu(X)\mu_0(A\backslash S)+\mu(S)\nu(A\cap S)=\mu(X)p_E(A)+(1-\mu(X))\nu(A).
\end{multline*}
So $\mu=\mu(X)p_E+(1-\mu(X))\nu$, as desired.

Part $(b)$. Suppose $\mu\in\fM^{\fim}_1(\cU)$. Then $\mu\in \fM^{\dfs}_1(\cU)$ and so, by Theorem \ref{thm:famnotfim} and the claim in part $(a)$, we have $\mu(x=b)>0$ for some $b\in\cU$. So $\fM_1^{\fim}(\cU)=\fM^{\triv}_1(\cU)$ by Proposition \ref{prop:localization}.
\end{proof}

\subsection{$K^r_s$-free hypergraphs}\label{sec:hyp}

We have now seen that if $s>r\geq 3$ then $T^r_s$ is dfs-trivial, while $T^2_s$ is not dfs-trivial for any $s>2$. The change in behavior from $r=2$ to $r\geq 3$ is reminiscent of a similar disparity at the level of dividing lines. In particular, $T^2$ is simple, but $T^2_s$ is not simple for any $s\geq 3$ (in fact, $T^2_s$ has SOP$_3$ by Shelah \cite{Sh500}). On the other hand, $T^r$ and $T^r_s$ are both simple for any $s>r\geq 3$ (this was shown by Hrushovski \cite{Udibook}; see also \cite[Section 7.1]{CoFA}). 

Despite the fact that $T^r_s$ is dfs-trivial for $s>r\geq 3$, we can  find interesting behavior in these theories at the level of $\phi$-types. First, let us recall some notions. Let $T$ be a complete theory with monster model $\cU$, and fix an $\cL$-formula $\phi(x;y)$. We let $S_\phi(\cU)$ be the space of complete $\phi$-types over $\cU$. Given $p\in S_\phi(\cU)$, we say:
\begin{enumerate}[$(	1)$]
\item $p$ is \textbf{definable} if the set $\{b\in\cU^y:\phi(x;b)\in p\}$ is definable (and thus, the same is true for any Boolean combination of $\phi(x;y_i)$);

\item $p$ is \textbf{finitely satisfiable in $M\prec\cU$} if any finite subset of $p$ is realized in $M$;
\item $p$ is \textbf{finitely approximated} if there is $M\prec\cU$ such that, for any formula $\psi(x;z)$, which is a finite Boolean combination of $\phi(x;y_i)$, and any $\epsilon>0$, there are $a_1,\ldots,a_n\in M^x$ such that, for any $c\in\cU^z$, $|p(\psi(x;c))-\Av_{\abar}(\psi(x;c))|<\epsilon$.
\end{enumerate}
In the last definition, we view $p$ as a $\{0,1\}$-valued local Keisler measure on $\phi$-formulas. Each of these notions extends naturally to the space of all local Keisler measures, analogous to Definition \ref{def:properties}. See \cite{GannNIP} for details. 

In \cite{GannNIP}, the second author proved a local version of the equivalence of $(i)$ and $(ii)$ in Theorem \ref{thm:HPS}. Specifically, if $\phi(x;y)$ is NIP and $\mu$ is a local Keisler measure on $\phi$-formulas, which is \df\ (suitably defined), then $\mu$ is finitely approximated. We will show that the analogue of this fails for simple formulas. In fact, we will find a complete $\phi$-type in a simple theory (specifically, $T^r_s$ for $s>r\geq 3$) that is \df\, but is not finitely approximated. Before defining this type, we  recall some results from graph theory. First, we state the following corollary of the \emph{Ramsey property} for  the class of finite $K^r_s$-free $r$-graphs.

\begin{theorem}[Ne\v{s}et\v{r}il \& R\"{o}dl 1979 \cite{NesRodFN,NesRodHyp}]\label{thm:NesRod}
Given $s>r\geq 2$ and $n\geq 1$, there is a finite $K^r_s$-free $r$-graph $G$ such that any edge-coloring of $G$ with $n$ colors admits a monochromatic copy of $K^r_{s-1}$. 
\end{theorem}

Next, we consider (vertex) colorings of weighted hypergraphs. In particular, given $r\geq 2$, a \emph{weighted $r$-graph} is a pair $H=(V,w)$ where $V$ is a finite vertex set and $w\colon [V]^r\to \R$ is a function (here $[V]^r$ is the set of $r$-element subsets of $V$). Suppose $H=(V,w)$ is a weighted $r$-graph. Set $w(V)=\sum_{\sigma\in [V]^r}w(\sigma)$. An \emph{$r$-coloring} of $H$ is a function $\chi\colon V\to[r]$. We say that an $r$-coloring $\chi$ \emph{splits} $\sigma\in [V]^r$ if $\chi(u)\neq\chi(v)$ for all distinct $u,v\in\sigma$. The \emph{weight} of an $r$-coloring $\chi$, denoted $w(\chi)$, is the sum of $w(\sigma)$ over all $\sigma\in [V]^r$ such that $\chi$ splits $\sigma$. The next fact is due to Erd\H{o}s and Kleitman \cite{EKcut} in the setting of unweighted hypergraphs.

\begin{lemma}\label{lem:maxcut}
Let $H=(V,w)$ be a finite weighted $r$-graph for some $r\geq 2$. Then there is an $r$-coloring $\chi$ of $H$ such that $w(\chi)\geq\frac{r!}{r^r}w(V)$.
\end{lemma}
\begin{proof}
Given an $r$-coloring $\chi$ of $H$ and $\sigma\in [V]^r$, let $w_\chi(\sigma)$ be $w(\sigma)$ if $\chi$ splits $\sigma$, and $0$ otherwise. So $w(\chi)=\sum_{\sigma\in [V]^r}w_\chi(\sigma)$. Let $n=|V|$. Then the number of $r$-colorings of $H$ is $r^n$ and, given $\sigma\in {[V]^r}$, the number of $r$-colorings of $H$ that  split $\sigma$ is $r^{n-r}r!$.  So we can compute the average weight of an $r$-coloring  of $H$ as follows: 
\[
\frac{1}{r^n}\sum_{\chi}w(\chi)=\frac{1}{r^n}\sum_{\chi}\sum_{\sigma}w_\chi(\sigma)=\frac{1}{r^n}\sum_{\sigma}\sum_{\chi}w_\chi(\sigma)=\frac{1}{r^n}\sum_{\sigma} \frac{r^{n}r!}{r^r}w(\sigma)=\frac{r!}{r^r}w(V).
\]
Therefore some $r$-coloring of $H$ has weight at least $\frac{r!}{r^r}w(V)$.
\end{proof}

Now we fix $s>r\geq 3$. Let $M\models T^r_s$ be the \Fraisse\ limit of the class of finite $K^r_{s}$-free $r$-graphs, and let $\cU\succ M$ be a sufficiently saturated elementary extension. Let $\phi(x_1,\ldots,x_{r-1};y)$ be the formula $\neg R(x_1,\ldots,x_{r-1},y)\wedge \bigwedge_{i\neq j}x_i\neq x_j$. We define $p_R\in S_\phi(\cU)$ to be the complete $\phi$-type containing $\phi(\xbar;b)$ for all $b\in\cU$. 

\begin{theorem}\label{thm:dfsnotfim}
The $\phi$-type $p_R$ is \df, but  not finitely approximated.
\end{theorem}
\begin{proof}
It is clear that $p_R$ is definable. We show that $p_R$ is finitely satisfiable in $M$. Fix $b_1,\ldots,b_n\in\cU$. We want to find $a_1,\ldots,a_{r-1}\in M$ such that $\neg R(\abar,b_i)$ holds for all $i\in[n]$, and $a_i\neq a_j$ for all distinct $i,j\in[r-1]$. 

By Theorem \ref{thm:NesRod}, there is a finite $K^{r-1}_s$-free $(r-1)$-graph $G=(W,E)$ such that any edge-coloring of $G$ with $n$ colors admits a monochromatic copy of $K^{r-1}_{s-1}$. Define an $r$-graph $H=(W,R)$ such that, given $\sigma\in [W]^r$, $R(\sigma)$ holds if and only if $[\sigma]^{r-1}\seq E$. Then $H$ is $K^r_s$-free since, if $A\in [W]^s$ is such that $[A]^r\seq R$ then $[A]^{r-1}\seq E$. So we may assume that $H$ is an induced subgraph of $M\backslash\{b_1,\ldots,b_n\}$. 

For $i\in[n]$, let $C_i=\{\tau\in [W]^{r-1}:R(\tau,b_i)\}$. Toward a contradiction, suppose $[W]^{r-1}= C_1\cup\ldots\cup C_n$. Then we can define an edge coloring $c\colon E\to [n]$ such that $c(\tau)=\min\{i\in [n]:\tau\in C_i\}$. By choice of $G$, there is $A\in[W]^{s-1}$ and $\ell\in [n]$ such that $[A]^{r-1}\seq E$ and $c(\tau)=\ell$ for all $\tau\in[A]^{r-1}$. But then $[A\cup\{b_\ell\}]^r\seq R$, contradicting that $\cU$ is $K^r_s$-free. So we may fix some $\sigma\in [W]^{r-1}\backslash (C_1\cup\ldots\cup C_n)$. Let $\sigma=\{a_1,\ldots,a_{r-1}\}$. Then $a_1,\ldots,a_{r-1}\in M$, $\neg R(\abar,b_i)$ for all $i\in[n]$, and $a_i\neq a_j$ for all distinct $i,j\in[r-1]$, as desired.

To show that $p_R$ is not finitely approximated, we fix $\abar^1,\ldots,\abar^n\in\cU^{r-1}$ and find some $b\in\cU$ such that $|\{t\in [n]:\phi(\abar^t;b)\}|< (1-\epsilon_r)n$, where $\epsilon_r=(r-1)^{1-r}(r-1)!$. After re-indexing if necessary, we may assume there is some $m\leq n$ such that, given $t\in[n]$, we have $|\{a^t_1,\ldots,a^t_{r-1}\}|=r-1$ if and only if $t\leq m$. 

Let $V=\{a^t_i:t\in[m],~i\in[r-1]\}$. For $\sigma\in [V]^{r-1}$, set
\[
I_\sigma=\left\{t\in [m]:\{a^t_1,\ldots,a^t_{r-1}\}=\sigma\right\}.
\]
 Define the weight function $w\colon [V]^{r-1}\to\{0,1,\ldots,m\}$ such that $w(\sigma)=|I_\sigma|$. Note that $\{I_\sigma:\sigma\in [V]^{r-1}\}$ is a partition of $[m]$ (with some $I_\sigma$ possibly empty), and so $w(V)=m$. By Lemma \ref{lem:maxcut}, there is an $(r-1)$-coloring $\chi$ of $(V,w)$ such that $w(\chi)\geq\epsilon_r m$. Let $\Sigma=\{\sigma\in[V]^{r-1}:\text{$\chi$ splits $\sigma$}\}$. 
 
 We now define an $r$-graph $(V',R')$  extending $(V,R)$. Let $V'=V\cup\{v_*\}$, where $v_*$ is a vertex not in $V$, and set $R'=R\cup\{\sigma\cup\{v_*\}:\sigma\in\Sigma\}$. Toward a contradiction, suppose $(V',R')$ is not $K^r_s$-free. Then there is $A\in[V']^s$ such that $[A]^r\seq R'$. So $v_*\in A$ since $(V,R)$ is $K^r_s$-free. Since $|A\cap V|=s-1>r-1$, there are distinct $v_1,v_2\in A\cap V$ such that $\chi(v_1)=\chi(v_2)$. Fix $\sigma\in [A\cap V]^{r-1}$ such that $v_1,v_2\in \sigma$. Then $\chi$ does not split $\sigma$, and so $\sigma\cup\{v_*\}\not\in R'$, which contradicts $[A]^r\seq R'$. 
 
 Finally, since $(V',R')$ is $K^r_s$-free, it follows that there is some $b\in\cU$ such that, given $\sigma\in [V]^{r-1}$,  $R(\sigma,b)$ holds if and only if $\sigma\in\Sigma$. Let $I=\{t\in[m]:R(\abar^t,b)\}$. Then $I=\bigcup_{\sigma\in\Sigma}I_{\sigma}$, and so $|I|=w(\chi)\geq \epsilon_rm$. So 
\[
|\{t\in [n]:\neg\phi(\abar^t;b)\}|= |I\cup \{m+1,\ldots,n\}|\geq \epsilon_r m+n-m\geq\epsilon_r  n,
\]
as desired. 
\end{proof}

Note that $p_R$ does not extend to a \emph{global} \df\ measure, since $T^r_s$ is dfs-trivial and $p_R$ cannot be extended to a global trivial measure.

\begin{remark}
The main reason to use hypergraphs in the above arguments was to work in a simple theory. However, a similar situation could be constructed in the theory $T^2_s$ for $s\geq 4$. Specifically, let $\phi(x,y;z)$ be  $\neg(E(x,z)\wedge E(y,z))\wedge x\neq y$, and let $p\in S_\phi(\cU)$ be the complete $\phi$-type containing $\phi(x,y;b)$ for all $b\in\cU$. Then an argument similar to the $r=3$ case of Theorem \ref{thm:dfsnotfim} shows that $p$ is \df, but not finitely approximated.
\end{remark}

In light of all of the examples above, we make the following conjecture and ask some questions.

\begin{conjecture} 
There is a theory $T$, and a Keisler measure $\mu\in\fM_x(\cU)$ such that $\mu$ is \df, but not finitely approximated. 
\end{conjecture}

\begin{question} 
Is there a simple (or even $\NTP_2$) theory $T$ and a \emph{global} Keisler measure $\mu\in\fM_x(\cU)$ such that either $\mu$ is \df\ but not finitely approximated, or $\mu$ is finitely approximated but not a frequency interpretation measure? Is there a type in a simple (or $\NTP_2$) theory with either of these properties?
\end{question}

\bibliographystyle{amsplain}

\begin{thebibliography}{1}

\bibitem{ACPgs}
Hans Adler, Enrique Casanovas, and Anand Pillay, \emph{Generic stability and
  stability}, J. Symb. Log. \textbf{79} (2014), no.~1, 179--185. \MR{3226018}


\bibitem{AKSR3t}
Mikl\'{o}s Ajtai, J\'{a}nos Koml\'{o}s, and Endre Szemer\'{e}di, \emph{A note
  on {R}amsey numbers}, J. Combin. Theory Ser. A \textbf{29} (1980), no.~3,
  354--360. \MR{600598}
  
 
\bibitem{BauGV}
Andreas Baudisch, \emph{Generic variations of models of {$T$}}, J. Symbolic
  Logic \textbf{67} (2002), no.~3, 1025--1038. \MR{1925955}
  
 \bibitem{Cabook}
Enrique Casanovas, \emph{Simple theories and hyperimaginaries}, Lecture Notes
  in Logic, vol.~39, Association for Symbolic Logic, Chicago, IL, 2011.
  \MR{2814891}

\bibitem{ChKa}
Artem Chernikov and Itay Kaplan, \emph{Forking and dividing in {${\rm NTP}\sb
  2$} theories}, J. Symbolic Logic \textbf{77} (2012), no.~1, 1--20.
  \MR{2951626}
 

\bibitem{ChRa}
Artem Chernikov and Nicholas Ramsey, \emph{On model-theoretic tree properties},
  J. Math. Log. \textbf{16} (2016), no.~2, 1650009, 41. \MR{3580894}
  
\bibitem{ChStNIP}
Artem Chernikov and Sergei Starchenko, \emph{Definable regularity lemmas for
  {NIP} hypergraphs}, arXiv:1607.07701, 2016.
  
  \bibitem{Co13}
Gabriel Conant, \emph{Forking and dividing in {H}enson graphs}, Notre Dame J.
  Form. Log. \textbf{58} (2017), no.~4, 555--566.
  
  \bibitem{CoFA}
\bysame, \emph{An axiomatic approach to free amalgamation}, J. Symbolic
  Logic \textbf{82} (2017), no.~2, 648–--671.
  
 
 \bibitem{DobGV}
Jan Dobrowolski, \emph{Generic variations and {$\rm NTP_1$}}, Arch. Math. Logic
  \textbf{57} (2018), no.~7-8, 861--871. \MR{3850687}

\bibitem{Duret}
Jean-Louis Duret, \emph{Les corps faiblement alg\'{e}briquement clos non
  s\'{e}parablement clos ont la propri\'{e}t\'{e} d'ind\'{e}pendence}, Model
  theory of algebra and arithmetic ({P}roc. {C}onf., {K}arpacz, 1979), Lecture
  Notes in Math., vol. 834, Springer, Berlin-New York, 1980, pp.~136--162.
  \MR{606784}

\bibitem{ErdR3t}
P.~Erd\H{o}s, \emph{Remarks on a theorem of {R}amsay}, Bull. Res. Council
  Israel. Sect. F \textbf{7F} (1957/1958), 21--24. \MR{0104594}
  
 \bibitem{EKcut}
Paul Erd\H{o}s and Daniel~J. Kleitman, \emph{On coloring graphs to maximize the
  proportion of multicolored {$k$}-edges}, J. Combinatorial Theory \textbf{5}
  (1968), 164--169. \MR{0228375}

  
\bibitem{ErRoR3t}
P.~Erd\H{o}s and C.~A. Rogers, \emph{The construction of certain graphs},
  Canad. J. Math. \textbf{14} (1962), 702--707. \MR{0141612}

\bibitem{GannNIP}
Kyle Gannon, \emph{Local {K}eisler measures and {NIP} formulas},
  arXiv:1811.02139, 2018.
  
  \bibitem{GOU}
Dar\'{i}o Garc\'{i}a, Alf Onshuus, and Alexander Usvyatsov, \emph{Generic
  stability, forking, and thorn-forking}, Trans. Amer. Math. Soc. \textbf{365}
  (2013), no.~1, 1--22. \MR{2984050}

\bibitem{Udibook}
Ehud Hrushovski, \emph{Pseudo-finite fields and related structures}, Model
  theory and applications, Quad. Mat., vol.~11, Aracne, Rome, 2002,
  pp.~151--212. \MR{2159717 (2006d:03059)}

  \bibitem{HPP}
Ehud Hrushovski, Ya'acov Peterzil, and Anand Pillay, \emph{Groups, measures,
  and the {NIP}}, J. Amer. Math. Soc. \textbf{21} (2008), no.~2, 563--596.
  \MR{2373360 (2008k:03078)}

  
  \bibitem{HP}
Ehud Hrushovski and Anand Pillay, \emph{On {NIP} and invariant measures}, J.
  Eur. Math. Soc. (JEMS) \textbf{13} (2011), no.~4, 1005--1061. \MR{2800483}
  
  \bibitem{HPS}
Ehud Hrushovski, Anand Pillay, and Pierre Simon, \emph{Generically stable and
  smooth measures in {NIP} theories}, Trans. Amer. Math. Soc. \textbf{365}
  (2013), no.~5, 2341--2366. \MR{3020101}
  
  
\bibitem{Keisler}
H.~Jerome Keisler, \emph{Measures and forking}, Ann. Pure Appl. Logic
  \textbf{34} (1987), no.~2, 119--169. \MR{890599}


\bibitem{KimR3t}
Jeong~Han Kim, \emph{The {R}amsey number {$R(3,t)$} has order of magnitude
  {$t^2/\log t$}}, Random Structures Algorithms \textbf{7} (1995), no.~3,
  173--207. \MR{1369063}

\bibitem{NesRodFN}
Jaroslav Ne\v{s}et\v{r}il and Vojt\v{e}ch R\"{o}dl, \emph{The {R}amsey property
  for graphs with forbidden complete subgraphs}, J. Combinatorial Theory Ser. B
  \textbf{20} (1976), no.~3, 243--249. \MR{0412004}
  
\bibitem{NesRodHyp}
\bysame, \emph{Ramsey theorem for
  classes of hypergraphs with forbidden complete subhypergraphs}, Czechoslovak
  Math. J. \textbf{29(104)} (1979), no.~2, 202--218. \MR{529509}
  
   \bibitem{Pibook}
Anand Pillay, \emph{An introduction to stability theory}, Oxford Logic Guides,
  vol.~8, The Clarendon Press Oxford University Press, New York, 1983.
  \MR{719195 (85i:03104)}

\bibitem{PiTa}
Anand Pillay and Predrag Tanovi\'c, \emph{Generic stability, regularity, and
  quasiminimality}, Models, logics, and higher-dimensional categories, CRM
  Proc. Lecture Notes, vol.~53, Amer. Math. Soc., Providence, RI, 2011,
  pp.~189--211. \MR{2867971}

\bibitem{Sh500}
Saharon Shelah, \emph{Toward classifying unstable theories}, Ann. Pure Appl.
  Logic \textbf{80} (1996), no.~3, 229--255. \MR{1402297 (97e:03052)}


\bibitem{ShGS}
\bysame, \emph{Classification theory for elementary classes with the
  dependence property---a modest beginning}, Sci. Math. Jpn. \textbf{59}
  (2004), no.~2, 265--316, Special issue on set theory and algebraic model
  theory. \MR{2062198}


\bibitem{Sibook}
Pierre Simon, \emph{A guide to {NIP} theories}, Lecture Notes in Logic,
  vol.~44, Association for Symbolic Logic, Chicago, IL; Cambridge Scientific
  Publishers, Cambridge, 2015. \MR{3560428} 
  
\bibitem{SiD}
\bysame, \emph{Distal and non-distal {NIP} theories}, Ann. Pure Appl.
  Logic \textbf{164} (2013), no.~3, 294--318. \MR{3001548}
  

\bibitem{SimGSG}
\bysame, \emph{On amalgamation in $\textnormal{NTP}_2$ theories and
  generically simple generics}, arXiv:1708.00182, 2017.
  

\bibitem{UsvyGS}
Alexander Usvyatsov, \emph{On generically stable types in dependent theories},
  J. Symbolic Logic \textbf{74} (2009), no.~1, 216--250. \MR{2499428}
  



\end{thebibliography}

\end{document}